\documentclass{CoreDP}

\usepackage{CustomMath}
\usepackage{TheoremParts}
\usepackage{SimpleAlgorithm}
\usepackage[style=trad-abbrv,citestyle=numeric-comp]{biblatex}
\usepackage{hyperref}
\usepackage[capitalize]{cleveref}

\bibliography{References}

\crefname{equation}{}{}

\makeatletter
\newcommand{\crefName}[1]{\@nameuse{cref@#1@name}}
\newcommand{\crefNamePlural}[1]{\@nameuse{cref@#1@name@plural}}
\newcommand{\CrefName}[1]{\@nameuse{Cref@#1@name}}
\newcommand{\CrefNamePlural}[1]{\@nameuse{Cref@#1@name@plural}}
\makeatother

\PaperGeneralInfo{%
  title={Subgradient Ellipsoid Method for Nonsmooth Convex Problems},%
  date={June 24, 2021},%
  CoreDPYear={2021},%
  CoreDPIssue={08}%
}

\AddPaperAuthor{%
  name={Anton Rodomanov},%
  affiliation={%
    Institute of Information and Communication Technologies, Electronics and
    Applied Mathematics~(ICTEAM), Catholic University of Louvain~(UCL),
    Louvain-la-Neuve, Belgium%
  },%
  email={anton.rodomanov@uclouvain.be}%
}

\AddPaperAuthor{%
  name={Yurii Nesterov},%
  affiliation={%
    Center for Operations Research and Econometrics~(CORE), Catholic University
    of Louvain~(UCL), Louvain-la-Neuve, Belgium%
  },%
  email={yurii.nesterov@uclouvain.be}%
}

\PaperAbstract{%
  In this paper, we present a new ellipsoid-type algorithm for solving nonsmooth
  problems with convex structure. Examples of such problems include nonsmooth
  convex minimization problems, convex-concave saddle-point problems and
  variational inequalities with monotone operator. Our algorithm can be seen as
  a combination of the standard Subgradient and Ellipsoid methods. However, in
  contrast to the latter one, the proposed method has a reasonable convergence
  rate even when the dimensionality of the problem is sufficiently large. For
  generating accuracy certificates in our algorithm, we propose an efficient
  technique, which ameliorates the previously known
  recipes~\cite{Nemirovski2010AccCert}.%
}

\PaperKeywords{%
  subgradient method, ellipsoid method, accuracy certificates, separating
  oracle, convex optimization, nonsmooth optimization, saddle-point problems,
  variational inequalities%
}

\PaperThanks{%
  This project has received funding from the European Research Council (ERC)
  under the European Union's Horizon 2020 research and innovation programme
  (grant agreement No.~788368).%
}

\begin{document}

\PrintTitleAndAbstract

%%%%%%%%%%%%%%%%%%%%%%%%%%%%%%%%%%%%%%%%%%%%%%%%%%%%%%%%%%%%%%%%%%%%%%%%%%%%%%%%
%%%%%%%%%%%%%%%%%%%%%%%%%%%%%%%%%%%%%%%%%%%%%%%%%%%%%%%%%%%%%%%%%%%%%%%%%%%%%%%%
\section{Introduction}
%%%%%%%%%%%%%%%%%%%%%%%%%%%%%%%%%%%%%%%%%%%%%%%%%%%%%%%%%%%%%%%%%%%%%%%%%%%%%%%%
%%%%%%%%%%%%%%%%%%%%%%%%%%%%%%%%%%%%%%%%%%%%%%%%%%%%%%%%%%%%%%%%%%%%%%%%%%%%%%%%

The Ellipsoid Method is a classical algorithm in Convex Optimization. It was
proposed in~1976 by Yudin and Nemirovski~\cite{YudinNemirovski1976Book} as the
modified method of centered cross-sections and then independently rediscovered a
year later by Shor~\cite{Shor1977CutOffMethod} in the form of the subgradient
method with space dilation. However, the popularity came to the Ellipsoid Method
only when Khachiyan used it in~1979 for proving his famous result on polynomial
solvability of Linear Programming~\cite{Khachiyan1979PolynomialLP}. Shortly
after, several polynomial algorithms, based on the Ellipsoid Method, were
developed for some combinatorial optimization
problems~\cite{Grotschel1981CombinatorialOptimization}. For more details and
historical remarks on the Ellipsoid Method,
see~\cite{%
  Bland1981EllipsoidSurvey,%
  Nemirovski1995BookInfoCompl,%
  BenTalNemirovski2021BookModernCO%
}.

Despite its long history, the Ellipsoid Method still has some issues which have
not been fully resolved or have been resolved only recently. One of them is the
computation of accuracy certificates which is important for generating
approximate solutions to dual problems or for solving general problems with
convex structure (saddle-point problems, variational inequalities, etc.). For a
long time, the procedure for calculating an accuracy certificate in the
Ellipsoid Method required solving an auxiliary piecewise linear optimization
problem (see, e.g., \crefNamePlural{section}~5 and~6
in~\cite{Nemirovski1995BookInfoCompl}). Although this auxiliary computation did
not use any additional calls to the oracle, it was still computationally
expensive and, in some cases, could take even more time than the Ellipsoid
Method itself. Only recently an efficient alternative has been proposed
\cite{Nemirovski2010AccCert}.

Another issue with the Ellipsoid Method is related to its poor dependency on the
dimensionality of the problem. Consider, e.g., the minimization problem
\begin{equation}\label{opt-prob}
  \min_{x \in Q} f(x),
\end{equation}
where $\Map{f}{\RealField^n}{\RealField}$ is a convex function and $Q
\DefinedEqual \SetBuilder{x \in \RealField^n}{\Norm{x} \leq R}$ is the Euclidean
ball of radius $R > 0$. The Ellipsoid Method for solving~\eqref{opt-prob} can be
written as follows (see, e.g., \crefName{section}~3.2.8
in~\cite{Nesterov2018Book}):
\begin{equation}\label{st-ell-met}
  \begin{aligned}
    x_{k+1} &\DefinedEqual x_k - \frac{1}{n + 1}
      \frac{W_k g_k}{\DualPairing{g_k}{W_k g_k}^{1/2}}, \\
    W_{k+1} &\DefinedEqual \frac{n^2}{n^2 - 1} \biggl(
      W_k - \frac{2}{n + 1}
        \frac{W_k g_k g_k\Transpose W_k}{\DualPairing{g_k}{W_k g_k}}
    \biggr),
    \qquad
    k \geq 0,
  \end{aligned}
\end{equation}
where $x_0 \DefinedEqual 0$, $W_0 \DefinedEqual R^2 I$ ($I$ is the identity
matrix) and $g_k \DefinedEqual f'(x_k)$ is an arbitrary nonzero subgradient if
$x_k \in Q$, and $g_k \DefinedEqual x_k$ is a separator of $x_k$ from $Q$ if
$x_k \notin Q$.

To solve problem~\eqref{opt-prob} with accuracy $\epsilon > 0$ (in terms of the
function value), the Ellipsoid Method needs
\begin{equation}\label{ell-met-compl}
  \BigO\Bigl( n^2 \ln \frac{M R}{\epsilon} \Bigr)
\end{equation}
iterations, where $M > 0$ is the Lipschitz constant of $f$ on $Q$ (see
\crefName{theorem}~3.2.11 in~\cite{Nesterov2018Book}). Looking at this estimate,
we can see an immediate drawback: it directly depends on the dimension and
becomes useless when $n \to \infty$. In particular, we cannot guarantee any
reasonable rate of convergence for the Ellipsoid Method when the dimensionality
of the problem is sufficiently big.

Note that the aforementioned drawback is an artifact of the method itself, not
its analysis. Indeed, when $n \to \infty$, iteration~\eqref{st-ell-met} reads
\[
  x_{k+1} \DefinedEqual x_k,
  \quad
  W_{k+1} \DefinedEqual W_k,
  \quad
  k \geq 0.
\]
Thus, the method stays at the same point and does not make any progress.

On the other hand, the simplest Subgradient Method for solving~\eqref{opt-prob}
possesses the ``dimension-independent'' $\BigO(M^2 R^2 / \epsilon^2)$ iteration
complexity bound (see, e.g., \crefName{section}~3.2.3 in
\cite{Nesterov2018Book}). Comparing this estimate with~\eqref{ell-met-compl}, we
see that the Ellipsoid Method is significantly faster than the Subgradient
Method only when $n$ is not too big compared to $\frac{M R}{\epsilon}$ and
significantly slower otherwise. Clearly, this situation is strange because the
former algorithm does much more work at every iteration by ``improving'' the
``metric'' $W_k$ which is used for measuring the norm of the subgradients.

In this paper, we propose a new ellipsoid-type algorithm for solving nonsmooth
problems with convex structure, which does not have the discussed above
drawback. Our algorithm can be seen as a combination of the Subgradient and
Ellipsoid methods and its convergence rate is basically as good as the best of
the corresponding rates of these two methods (up to some logarithmic factors).
In particular, when $n \to \infty$, the convergence rate of our algorithm
coincides with that of the Subgradient Method.

\subsection*{Contents}

This paper is organized as follows. In \cref{sec-conv-prob-review}, we review
the general formulation of a problem with convex structure and the associated
with it notions of \emph{accuracy certificate} and \emph{residual}. Our
presentation mostly follows~\cite{Nemirovski2010AccCert} with
examples taken from~\cite{Nesterov2009PrimalDual}. Then, in
\cref{sec-semicert-gap}, we introduce the notions of \emph{accuracy
semicertificate} and \emph{gap} and discuss their relation with those of
accuracy certificate and residual.

In \cref{sec-gen-alg-sch}, we present the general algorithmic scheme of our
methods. To measure the convergence rate of this scheme, we introduce the notion
of \emph{sliding gap} and establish some preliminary bounds on it.

In \cref{sec-main-inst}, we discuss different choices of parameters in our
general scheme. First, we show that, by setting some of the parameters to zero,
we obtain the standard Subgradient and Ellipsoid methods. Then we consider a
couple of other less trivial choices which lead to two new algorithms. The
principal of these new algorithms is the latter one, which we call the
\emph{Subgradient Ellipsoid Method}. We demonstrate that the convergence rate of
this algorithm is basically as good as the best of those of the Subgradient and
Ellipsoid methods.

In \cref{sec-gen-acc-cert}, we show that, for both our new methods, it is
possible to efficiently generate accuracy semicertificates whose gap is upper
bounded by the sliding gap. We also compare our approach with the recently
proposed technique from~\cite{Nemirovski2010AccCert} for building
accuracy certificates for the standard Ellipsoid Method.

In \cref{sec-impl-dets}, we discuss how to efficiently implement our general
scheme and the procedure for generating accuracy semicertificates. In
particular, we show that the time and memory requirements of our scheme are the
same as in the standard Ellipsoid Method.

Finally, in \cref{sec-concl}, we discuss some open questions.

\subsection*{Notation and Generalities}

In this paper, $\VectorSpace{E}$ denotes an arbitrary $n$-dimensional real
vector space. Its dual space, composed of all linear functionals on
$\VectorSpace{E}$, is denoted by~$\VectorSpace{E}\Dual$. The value of $s \in
\VectorSpace{E}\Dual$, evaluated at $x \in \VectorSpace{E}$, is denoted by
$\DualPairing{s}{x}$.

Let us introduce in the spaces $\VectorSpace{E}$ and $\VectorSpace{E}\Dual$ a
pair of conjugate Euclidean norms. For this, we fix some self-adjoint positive
definite linear operator $\Map{B}{\VectorSpace{E}}{\VectorSpace{E}\Dual}$ and
define
\[
  \Norm{x} \DefinedEqual \DualPairing{B x}{x}^{1/2},
  \quad x \in \VectorSpace{E},
  \qquad\quad
  \DualNorm{s} \DefinedEqual \DualPairing{s}{B^{-1} s}^{1/2},
  \quad s \in \VectorSpace{E}\Dual.
\]
Note that, for any $s \in \VectorSpace{E}\Dual$ and $x \in \VectorSpace{E}$, we
have the Cauchy-Schwarz inequality
\[
  \AbsoluteValue{\DualPairing{s}{x}} \leq \DualNorm{s} \Norm{x},
\]
which becomes an equality if and only if $s$ and $B x$ are collinear. In
addition to $\Norm{\cdot}$ and $\DualNorm{\cdot}$, we often work with other
Euclidean norms defined in the same way but using another reference operator
instead of $B$. In this case, we write $\RelativeNorm{\cdot}{G}$ and
$\RelativeDualNorm{\cdot}{G}$, where
$\Map{G}{\VectorSpace{E}}{\VectorSpace{E}\Dual}$ is the corresponding
self-adjoint positive definite linear operator.

Sometimes, in the formulas, involving products of linear operators, it is
convenient to treat $x \in \VectorSpace{E}$ as a linear operator from
$\RealField$ to $\VectorSpace{E}$, defined by $x \alpha \DefinedEqual \alpha x$,
and $x\Adjoint$ as a linear operator from $\VectorSpace{E}\Dual$ to
$\RealField$, defined by $x\Adjoint s \DefinedEqual \DualPairing{s}{x}$.
Likewise, any $s \in \VectorSpace{E}\Dual$ can be treated as a linear operator
from $\RealField$ to $\VectorSpace{E}\Dual$, defined by $s \alpha \DefinedEqual
\alpha s$, and $s\Adjoint$ as a linear operator from $\VectorSpace{E}$ to
$\RealField$, defined by $s\Adjoint x \DefinedEqual \DualPairing{s}{x}$. Then,
$x x\Adjoint$ and $s s\Adjoint$ are rank-one self-adjoint linear operators from
$\VectorSpace{E}\Dual$ to $\VectorSpace{E}$ and from $\VectorSpace{E}$ to
$\VectorSpace{E}\Dual$ respectively, acting as follows: $(x x\Adjoint) s =
\DualPairing{s}{x} x$ and $(s s\Adjoint) x = \DualPairing{s}{x} s$ for any $x
\in \VectorSpace{E}$ and $s \in \VectorSpace{E}\Dual$.

For a self-adjoint linear operator
$\Map{G}{\VectorSpace{E}}{\VectorSpace{E}\Dual}$, by $\Trace G$ and $\det G$, we
denote the trace and determinant of $G$ with respect to our fixed operator $B$:
\[
  \Trace G \DefinedEqual \Trace(B^{-1} G),
  \qquad
  \det G \DefinedEqual \det(B^{-1} G).
\]
Note that, in these definitions, $B^{-1} G$ is a linear operator from
$\VectorSpace{E}$ to $\VectorSpace{E}$, so $\Trace(B^{-1} G)$ and $\det(B^{-1}
G)$ are the standard well-defined notions of trace and determinant of a linear
operator acting on the same space. For example, they can be defined as the trace
and determinant of the matrix representation of $B^{-1} G$ with respect to an
arbitrary chosen basis in $\VectorSpace{E}$ (the result is independent of the
particular choice of basis). Alternatively, $\Trace G$ and $\det G$ can be
equivalently defined as the sum and product, respectively, of the eigenvalues of
$G$ with respect to $B$.

For a point $x \in \VectorSpace{E}$ and a real $r > 0$, by
\[
  \Ball{x}{r} \DefinedEqual
  \SetBuilder{y \in \VectorSpace{E}}{\Norm{y - x} \leq r},
\]
we denote the closed Euclidean ball with center $x$ and radius $r$.

Given two solids\footnote{%
  Hereinafter, a \emph{solid} is any convex compact set with nonempty interior.%
} %
$Q, Q_0 \subseteq \VectorSpace{E}$, we can define the \emph{relative volume} of
$Q$ with respect to~$Q_0$ by $\Volume(Q / Q_0) \DefinedEqual \Volume Q^e /
\Volume Q_0^e$, where $e$ is an arbitrary basis in $\VectorSpace{E}$, $Q^e,
Q_0^e \subseteq \RealField^n$ are the coordinate representations of the sets $Q,
Q_0$ in the basis $e$ and $\Volume$ is the Lebesgue measure in $\RealField^n$.
Note that the relative volume is independent of the particular choice of the
basis $e$. Indeed, for any other basis $f$, we have $Q^e = T_f^e Q^f$, $Q_0^e =
T_f^e Q_0^f$, where $T_f^e$ is the $n \times n$ change-of-basis matrix, so
$\Volume Q^e = (\det T_f^e) (\Volume Q^f)$, $\Volume Q_0^e = (\det T_f^e)
(\Volume Q_0^f)$ and hence $\Volume Q^e / \Volume Q_0^e = \Volume Q^f / \Volume
Q_0^f$.

For us, it will be convenient to define the \emph{volume} of a solid $Q
\subseteq \VectorSpace{E}$ as the relative volume of $Q$ with respect to the
unit ball:
\[
  \Volume Q \DefinedEqual \Volume(Q / \Ball{0}{1}).
\]
For an ellipsoid $W \DefinedEqual \SetBuilder{x \in
\VectorSpace{E}}{\DualPairing{G x}{x} \leq 1}$, where
$\Map{G}{\VectorSpace{E}}{\VectorSpace{E}\Dual}$ is a self-adjoint positive
definite linear operator, we have $\Volume W = (\det G)^{-1/2}$.

%%%%%%%%%%%%%%%%%%%%%%%%%%%%%%%%%%%%%%%%%%%%%%%%%%%%%%%%%%%%%%%%%%%%%%%%%%%%%%%%
%%%%%%%%%%%%%%%%%%%%%%%%%%%%%%%%%%%%%%%%%%%%%%%%%%%%%%%%%%%%%%%%%%%%%%%%%%%%%%%%
\section{Convex Problems and Accuracy Certificates}\label{sec-conv-probs}
%%%%%%%%%%%%%%%%%%%%%%%%%%%%%%%%%%%%%%%%%%%%%%%%%%%%%%%%%%%%%%%%%%%%%%%%%%%%%%%%
%%%%%%%%%%%%%%%%%%%%%%%%%%%%%%%%%%%%%%%%%%%%%%%%%%%%%%%%%%%%%%%%%%%%%%%%%%%%%%%%

\subsection{Description and Examples}\label{sec-conv-prob-review}

In this paper, we consider numerical algorithms for solving \emph{problems with
convex structure}. The main examples of such problems are convex minimization
problems, convex-concave saddle-point problems, convex Nash equilibrium
problems, and variational inequalities with monotone operators.

The general formulation of a problem with convex structure involves two objects:
\begin{itemize}
  \item Solid $Q \subseteq \VectorSpace{E}$ (called the \emph{feasible set}),
  represented by the \emph{Separation Oracle}: given any point $x \in
  \VectorSpace{E}$, this oracle can check whether $x \in \Interior Q$, and if
  not, it reports a vector $g_Q(x) \in \VectorSpace{E}\Dual \setminus \Set{0}$
  which separates $x$ from $Q$:
  \begin{equation}\label{def-sep}
    \DualPairing{g_Q(x)}{x - y} \geq 0,
    \quad \forall y \in Q.
  \end{equation}

  \item Vector field $\Map{g}{\Interior Q}{\VectorSpace{E}\Dual}$, represented
  by the \emph{First-Order Oracle}: given any point $x \in \Interior Q$, this
  oracle returns the vector $g(x)$.
\end{itemize}
In what follows, we only consider the problems satisfying the following
condition:
\begin{equation}\label{gen-conv-x-star}
  \exists x^* \in Q \colon \quad
  \DualPairing{g(x)}{x - x^*} \geq 0,
  \quad \forall x \in \Interior Q.
\end{equation}

A numerical algorithm for solving a problem with convex structure starts at some
point $x_0 \in \VectorSpace{E}$. At each step $k \geq 0$, it queries the oracles
at the current \emph{test point} $x_k$ to obtain the new information about the
problem, and then somehow uses this new information to form the next test point
$x_{k+1}$. Depending on whether $x_k \in \Interior Q$, the $k$th step of
the algorithm is called \emph{productive} or \emph{nonproductive}.

The total information, obtained by the algorithm from the oracles after $k \geq
1$ steps, comprises its \emph{execution protocol} which consists of:
\begin{itemize}
  \item The test points $x_0, \ldots, x_{k-1} \in \VectorSpace{E}$.
  \item The set of productive steps $I_k \DefinedEqual \SetBuilder{0 \leq i \leq
  k-1}{x_i \in \Interior Q}$.
  \item The vectors $g_0, \ldots, g_{k-1} \in \VectorSpace{E}\Dual$ reported by
  the oracles: $g_i \DefinedEqual g(x_i)$, if $i \in I_k$, and $g_i
  \DefinedEqual g_Q(x_i)$, if $i \notin I_k$, $0 \leq i \leq k-1$.
  
\end{itemize}

An \emph{accuracy certificate}, associated with the above execution protocol, is
a nonnegative vector $\lambda \DefinedEqual (\lambda_0, \ldots, \lambda_{k-1})$
such that $S_k(\lambda) \DefinedEqual \sum_{i \in I_k} \lambda_i > 0$ (and, in
particular, $I_k \neq \emptyset$). Given any solid $\Omega$, containing $Q$, we
can define the following \emph{residual} of $\lambda$ on $\Omega$:
\begin{equation}\label{def-resid}
  \epsilon_k(\lambda)
  \DefinedEqual \max_{x \in \Omega} \frac{1}{S_k(\lambda)}
    \sum_{i=0}^{k-1} \lambda_i \DualPairing{g_i}{x_i - x},
\end{equation}
which is easily computable whenever $\Omega$ is a simple set (e.g., a Euclidean
ball). Note that
\begin{equation}\label{resid-lbd}
  \epsilon_k(\lambda)
  \geq \max_{x \in Q} \frac{1}{S_k(\lambda)} \sum_{i=0}^{k-1}
    \lambda_i \DualPairing{g_i}{x_i - x}
  \geq \max_{x \in Q} \frac{1}{S_k(\lambda)} \sum_{i \in I_k}
    \lambda_i \DualPairing{g_i}{x_i - x}
\end{equation}
and, in particular, $\epsilon_k(\lambda) \geq 0$ in view of
\cref{gen-conv-x-star}.

In what follows, we will be interested in the algorithms, which can produce
accuracy certificates $\lambda\UpperIndex{k}$ with
$\epsilon_k(\lambda\UpperIndex{k}) \to 0$ at a certain rate. This is a
meaningful goal because, for all known instances of problems with convex
structure, the residual $\epsilon_k(\lambda)$ upper bounds a certain natural
inaccuracy measure for the corresponding problem. Let us briefly review some
standard examples (for more examples, see
\cite{Nesterov2009PrimalDual,Nemirovski2010AccCert} and the
references therein).

\begin{example}[Convex Minimization Problem]
  Consider the problem
  \begin{equation}\label{conv-min-prob}
    f^* \DefinedEqual \min_{x \in Q} f(x),
  \end{equation}
  where $Q \subseteq \VectorSpace{E}$ is a solid and
  $\Map{f}{\VectorSpace{E}}{\RealFieldPlusInfinity}$ is closed convex and finite
  on $\Interior Q$.
  
  The First-Order Oracle for~\eqref{conv-min-prob} is $g(x) \DefinedEqual
  f'(x)$, $x \in \Interior Q$, where $f'(x)$ is an arbitrary subgradient of $f$
  at $x$. Clearly, \cref{gen-conv-x-star} holds for $x^*$ being any solution
  of~\eqref{conv-min-prob}.
  
  It is not difficult to verify that, in this example, the residual
  $\epsilon_k(\lambda)$ upper bounds the functional residual: for $\hat{x}_k
  \DefinedEqual \frac{1}{S_k(\lambda)} \sum_{i \in I_k} \lambda_i x_i$ or $x_k^*
  \DefinedEqual \argmin \SetBuilder{f(x)}{x \in X_k}$, where $X_k \DefinedEqual
  \SetBuilder{x_i}{i \in I_k}$, we have $f(\hat{x}_k) - f^* \leq
  \epsilon_k(\lambda)$ and $f(x_k^*) - f^* \leq \epsilon_k(\lambda)$.
  
  Moreover, $\epsilon_k(\lambda)$, in fact, upper bounds the primal-dual gap for
  a certain dual problem for~\eqref{conv-min-prob}. Indeed, let
  $\Map{f_*}{\VectorSpace{E}\Dual}{\RealFieldPlusInfinity}$ be the conjugate
  function of $f$. Then, we can represent~\eqref{conv-min-prob} in the following
  dual form:
  \begin{equation}\label{conv-min-prob-dual}
    f^*
    = \adjustlimits \min_{x \in Q} \max_{s \in \EffectiveDomain f_*}
      [\DualPairing{s}{x} - f_*(s)]
    = \max_{s \in \EffectiveDomain f_*} [-f_*(s) - \xi_Q(-s)],
  \end{equation}
  where $\EffectiveDomain f_* \DefinedEqual \SetBuilder{s \in
  \VectorSpace{E}\Dual}{f_*(s) < +\infty}$ and $\xi_Q(-s) \DefinedEqual \max_{x
  \in Q} \DualPairing{-s}{x}$. Denote $s_k \DefinedEqual \frac{1}{S_k(\lambda)}
  \sum_{i \in I_k} \lambda_i g_i$. Then, using \cref{resid-lbd} and the
  convexity of $f$ and $f_*$, we obtain
  \[
    \begin{aligned}
      \epsilon_k(\lambda)
      &\geq \frac{1}{S_k(\lambda)} \sum_{i \in I_k}
        \lambda_i \DualPairing{g_i}{x_i} + \xi_Q(-s_k)
      = \frac{1}{S_k(\lambda)} \sum_{i \in I_k}
        \lambda_i [f(x_i) + f_*(g_i)] + \xi_Q(-s_k) \\
      &\geq f(\hat{x}_k) + f_*(s_k) + \xi_Q(-s_k).
    \end{aligned}
  \]
  Thus, $\hat{x}_k$ and $s_k$ are $\epsilon_k(\lambda)$-approximate solutions
  (in terms of function value) to problems~\eqref{conv-min-prob}
  and~\eqref{conv-min-prob-dual}, respectively. Note that the same is true if we
  replace $\hat{x}_k$ with $x_k^*$.
\end{example}

\begin{example}[Convex-Concave Saddle-Point Problem]
  Consider the following problem:
  \begin{equation}\label{saddle-prob}
    \text{Find $(u^*, v^*) \in U \times V$} \colon \quad
    f(u^*, v) \leq f(u^*, v^*) \leq f(u, v^*),
    \quad
    \forall (u, v) \in U \times V,
  \end{equation}
  where $U$, $V$ are solids in some finite-dimensional vector spaces
  $\VectorSpace{E}_u$, $\VectorSpace{E}_v$, respectively, and $\Map{f}{U \times
  V}{\RealField}$ is a continuous function which is \emph{convex-concave}, i.e.
  $f(\cdot, v)$ is convex and $f(u, \cdot)$ is concave for any $u \in U$ and any
  $v \in V$.

  In this example, we set $\VectorSpace{E} \DefinedEqual \VectorSpace{E}_u
  \times \VectorSpace{E}_v$, $Q \DefinedEqual U \times V$ and use the
  First-Order Oracle
  \[
    g(x) \DefinedEqual (f'_u(x), -f'_v(x)),
    \quad x \DefinedEqual (u, v) \in \Interior Q,
  \]  
  where $f'_u(x)$ is an arbitrary subgradient of $f(\cdot, v)$ at $u$ and
  $f'_v(y)$ is an arbitrary supergradient of $f(u, \cdot)$ at $v$. Then, for any
  $x \DefinedEqual (u, v) \in \Interior Q$ and any $x' \DefinedEqual (u', v')
  \in Q$,
  \begin{equation}\label{sad-prob-grad-ineq}
    \DualPairing{g(x)}{x - x'}
    = \DualPairing{f'_u(x)}{u - u'} - \DualPairing{f'_v(x)}{v - v'}
    \geq f(u, v') - f(u', v).
  \end{equation}
  In particular, \cref{gen-conv-x-star} holds for $x^* \DefinedEqual (u^*, v^*)$
  in view of \cref{saddle-prob}.

  Let $\Map{\phi}{U}{\RealField}$ and $\Map{\psi}{V}{\RealField}$ be the
  functions
  \[
    \phi(u) \DefinedEqual \max_{v \in V} f(u, v),
    \qquad
    \psi(v) \DefinedEqual \min_{u \in U} f(u, v).
  \]
  In view of \cref{saddle-prob}, we have $\psi(v) \leq f(u^*, v^*) \leq \phi(u)$
  for all $(u, v) \in U \times V$. Therefore, the difference $\phi(u) - \psi(v)$
  (called the \emph{primal-dual gap}) can be used for measuring the quality of
  an approximate solution $x \DefinedEqual (u, v) \in Q$ to
  problem~\eqref{saddle-prob}.
  
  Denoting $\hat{x}_k \DefinedEqual \frac{1}{S_k(\lambda)} \sum_{i \in I_k}
  \lambda_i x_i \EqualDefines (\hat{u}_k, \hat{v}_k)$ and using
  \cref{resid-lbd}, we obtain
  \[
    \begin{aligned}
	    \epsilon_k(\lambda)
      &\geq \max_{x \in Q} \frac{1}{S_k(\lambda)} \sum_{i \in I_k}
        \lambda_i \DualPairing{g_i}{x_i - x}
      \geq \max_{u \in U, v \in V} \frac{1}{S_k(\lambda)} \sum_{i \in I_k}
        \lambda_i [f(u_i, v) - f(u, v_i)] \\
      &\geq \max_{u \in U, v \in V} [f(\hat{u}_k, v) - f(u, \hat{v}_k)]
      = \phi(\hat{u}_k) - \psi(\hat{v}_k),
    \end{aligned}
  \]
  where the second inequality is due to \cref{sad-prob-grad-ineq} and the last
  one follows from the convexity-concavity of $f$. Thus, the residual
  $\epsilon_k(\lambda)$ upper bounds the primal-dual gap for the approximate
  solution $\hat{x}_k$.
\end{example}

\begin{example}[Variational Inequality with Monotone Operator]
  Let $Q \subseteq \VectorSpace{E}$ be a solid and let
  $\Map{V}{Q}{\VectorSpace{E}\Dual}$ be a continuous operator which is
  \emph{monotone}, i.e. $\DualPairing{V(x) - V(y)}{x - y} \geq 0$ for all $x, y
  \in Q$. The goal is to solve the following (weak) \emph{variational
  inequality}:
  \begin{equation}\label{vi-prob}
    \text{Find $x^* \in Q$} \colon \quad
    \DualPairing{V(x)}{x - x^*} \geq 0,
    \quad
    \forall x \in Q.
  \end{equation}
  Since $V$ is continuous, this problem is equivalent to its strong variant:
  find $x^* \in Q$ such that $\DualPairing{V(x^*)}{x - x^*} \geq 0$ for all $x
  \in Q$.

  A standard tool for measuring the quality of an approximate solution
  to~\eqref{vi-prob} is the \emph{dual gap function}, introduced
  in~\cite{Auslender1973}:
  \[
    f(x) \DefinedEqual \max_{y \in Q} \DualPairing{V(y)}{x - y},
    \qquad x \in Q.
  \]
  It is easy to see that $f$ is a convex nonnegative function which equals $0$
  exactly at the solutions of~\eqref{vi-prob}.

  In this example, the First-Order Oracle is defined by $g(x) \DefinedEqual
  V(x)$, $x \in \Interior Q$. Denote $\hat{x}_k \DefinedEqual
  \frac{1}{S_k(\lambda)} \sum_{i \in I_k} \lambda_i x_i$. Then, using
  \cref{resid-lbd} and the monotonicity of $V$, we obtain
  \[
    \epsilon_k(\lambda)
    \geq \max_{x \in Q} \frac{1}{S_k(\lambda)} \sum_{i \in I_k}
      \lambda_i \DualPairing{V(x_i)}{x_i - x}
    \geq \max_{x \in Q} \frac{1}{S_k(\lambda)} \sum_{i \in I_k}
      \lambda_i \DualPairing{V(x)}{x_i - x}
    = f(\hat{x}_k).
  \]
  Thus, $\epsilon_k(\lambda)$ upper bounds the dual gap function for the
  approximate solution $\hat{x}_k$.
\end{example}

%%%%%%%%%%%%%%%%%%%%%%%%%%%%%%%%%%%%%%%%%%%%%%%%%%%%%%%%%%%%%%%%%%%%%%%%%%%%%%%%
%%%%%%%%%%%%%%%%%%%%%%%%%%%%%%%%%%%%%%%%%%%%%%%%%%%%%%%%%%%%%%%%%%%%%%%%%%%%%%%%
\subsection{Establishing Convergence of Residual}
\label{sec-semicert-gap}
%%%%%%%%%%%%%%%%%%%%%%%%%%%%%%%%%%%%%%%%%%%%%%%%%%%%%%%%%%%%%%%%%%%%%%%%%%%%%%%%
%%%%%%%%%%%%%%%%%%%%%%%%%%%%%%%%%%%%%%%%%%%%%%%%%%%%%%%%%%%%%%%%%%%%%%%%%%%%%%%%

For the algorithms, considered in this paper, instead of accuracy certificates
and residuals, it turns out to be more convenient to speak about closely related
notions of \emph{accuracy semicertificates} and \emph{gaps}, which we now
introduce.

As before, let $x_0, \dots, x_{k-1}$ be the test points, generated by the
algorithm after $k \geq 1$ steps, and let $g_0, \dots, g_{k-1}$ be the
corresponding oracle outputs. An \emph{accuracy semicertificate}, associated
with this information, is a nonnegative vector $\lambda \DefinedEqual
(\lambda_0, \dots, \lambda_{k-1})$ such that $\Gamma_k(\lambda) \DefinedEqual
\sum_{i=0}^{k-1} \lambda_i \DualNorm{g_i} > 0$. Given any solid $\Omega$,
containing $Q$, the \emph{gap} of $\lambda$ on $\Omega$ is defined in the
following way:
\begin{equation}\label{def-gap}
  \delta_k(\lambda)
  \DefinedEqual \max_{x \in \Omega} \frac{1}{\Gamma_k(\lambda)}
      \sum_{i=0}^{k-1} \lambda_i \DualPairing{g_i}{x_i - x}.
\end{equation}
Comparing these definitions with those of accuracy certificate and residual, we
see that the only difference between them is that now we use a different
``normalizing'' coefficient: $\Gamma_k(\lambda)$ instead of $S_k(\lambda)$.
Also, in the definitions of semicertificate and gap, we do not make any
distinction between productive and nonproductive steps. Note that
$\delta_k(\lambda) \geq 0$.

Let us demonstrate that by making the gap sufficiently small, we can make the
corresponding residual sufficiently small as well. For this, we need the
following standard assumption about our problem with convex structure (see,
e.g., \cite{Nemirovski2010AccCert}).

\begin{assumption}\label{as-g-sbnd}%
  The vector field $g$, reported by the First-Order Oracle, is semibounded:
  \[
    \DualPairing{g(x)}{y - x} \leq V,
    \quad
    \forall x \in \Interior Q,
    \ \forall y \in Q.
  \]
\end{assumption}
A classical example of a semibounded field is a bounded one: if there is $M \geq
0$, such that $\DualNorm{g(x)} \leq M$ for all $x \in \Interior Q$, then $g$ is
semibounded with $V \DefinedEqual M D$, where $D$ is the diameter of $Q$.
However, there exist other examples. For instance, if $g$ is the subgradient
field of a convex function $\Map{f}{\VectorSpace{E}}{\RealFieldPlusInfinity}$,
which is finite and continuous on $Q$, then $g$ is semibounded with $V
\DefinedEqual \max_Q f - \min_Q f$ (variation of $f$ on~$Q$); however, $g$ is
not bounded if $f$ is not Lipschitz continuous (e.g., $f(x) \DefinedEqual
-\sqrt{x}$ on $Q \DefinedEqual \ClosedClosedInterval{0}{1}$). Another
interesting example is the subgradient field $g$ of a $\nu$-self-concordant
barrier $\Map{f}{\VectorSpace{E}}{\RealFieldPlusInfinity}$ for the set $Q$; in
this case, $g$ is semibounded with $V \DefinedEqual \nu$ (see e.g.
\cite[Theorem~5.3.7]{Nesterov2018Book}), while $f(x) \to +\infty$ at the
boundary of $Q$.

\begin{lemma}\label{lm-resid-ubd-via-gap}%
  Let $\lambda$ be a semicertificate such that $\delta_k(\lambda) < r$, where
  $r$ is the largest of the radii of Euclidean balls contained in $Q$. Then,
  $\lambda$ is a certificate and
  \[
    \epsilon_k(\lambda)
    \leq \frac{\delta_k(\lambda)}{r - \delta_k(\lambda)} V.
  \]
\end{lemma}

\begin{proof}
  Denote $\delta_k \DefinedEqual \delta_k(\lambda)$, $\Gamma_k \DefinedEqual
  \Gamma_k(\lambda)$, $S_k \DefinedEqual S_k(\lambda)$. Let $\bar{x} \in Q$ be
  such that $\Ball{\bar{x}}{r} \subseteq Q$. For each $0 \leq i \leq k-1$, let
  $z_i$ be a maximizer of $z \mapsto \DualPairing{g_i}{z - \bar{x}}$ on
  $\Ball{\bar{x}}{r}$. Then, for any $0 \leq i \leq k-1$, we have
  $\DualPairing{g_i}{\bar{x} - x_i} = \DualPairing{g_i}{z_i - x_i} - r
  \DualNorm{g_i}$ with $z_i \in Q$. Therefore,
  \begin{equation}\label{sum-g-x-bar}
    \sum_{i=0}^{k-1} \lambda_i \DualPairing{g_i}{\bar{x} - x_i}
    = \sum_{i=0}^{k-1} \lambda_i \DualPairing{g_i}{z_i - x_i} - r \Gamma_k
    \leq S_k V - r \Gamma_k,
  \end{equation}
  where the inequality follows from the separation property~\eqref{def-sep} and
  \cref{as-g-sbnd}.

  Let $x \in \Omega$ be arbitrary. Denoting $y \DefinedEqual \bigl( \delta_k
  \bar{x} + (r - \delta_k) x \bigr) / r \in \Omega$, we obtain
  \begin{equation}\label{cert-prel}    
    \begin{aligned}
      (r - \delta_k) \sum_{i=0}^{k-1} \lambda_i \DualPairing{g_i}{x_i - x}
      &= r \sum_{i=0}^{k-1} \lambda_i \DualPairing{g_i}{x_i - y}
        + \delta_k \sum_{i=0}^{k-1}
            \lambda_i \DualPairing{g_i}{\bar{x} - x_i} \\
      &\leq r \delta_k \Gamma_k +
        \delta_k \sum_{i=0}^{k-1} \lambda_i \DualPairing{g_i}{\bar{x} - x_i}
      \leq \delta_k S_k V.
    \end{aligned}
  \end{equation}
  where the inequalities follow from the definition~\eqref{def-gap} of
  $\delta_k$ and \cref{sum-g-x-bar}, respectively.

  It remains to show that $\lambda$ is a certificate, i.e. $S_k > 0$. But this
  is simple. Indeed, if $S_k = 0$, then, taking $x \DefinedEqual \bar{x}$ in
  \cref{cert-prel} and using \cref{sum-g-x-bar}, we get $0 \geq \sum_{i=0}^{k-1}
  \lambda_i \DualPairing{g_i}{x_i - \bar{x}} \geq r \Gamma_k$, which contradicts
  our assumption that $\lambda$ is a semicertificate, i.e. $\Gamma_k > 0$.
\end{proof}

According to \cref{lm-resid-ubd-via-gap}, from the convergence rate of the gap
$\delta_k(\lambda\UpperIndex{k})$ to zero, we can easily obtain the
corresponding convergence rate of the residual
$\epsilon_k(\lambda\UpperIndex{k})$. In particular, to ensure that
$\epsilon_k(\lambda\UpperIndex{k}) \leq \epsilon$ for some $\epsilon > 0$, it
suffices to make $\delta_k(\lambda\UpperIndex{k}) \leq \delta(\epsilon)
\DefinedEqual \frac{\epsilon r}{\epsilon + V}$. For this reason, in the rest of
this paper, we can focus our attention on studying the convergence rate only for
the gap.

%%%%%%%%%%%%%%%%%%%%%%%%%%%%%%%%%%%%%%%%%%%%%%%%%%%%%%%%%%%%%%%%%%%%%%%%%%%%%%%%
%%%%%%%%%%%%%%%%%%%%%%%%%%%%%%%%%%%%%%%%%%%%%%%%%%%%%%%%%%%%%%%%%%%%%%%%%%%%%%%%
\section{General Algorithmic Scheme}\label{sec-gen-alg-sch}
%%%%%%%%%%%%%%%%%%%%%%%%%%%%%%%%%%%%%%%%%%%%%%%%%%%%%%%%%%%%%%%%%%%%%%%%%%%%%%%%
%%%%%%%%%%%%%%%%%%%%%%%%%%%%%%%%%%%%%%%%%%%%%%%%%%%%%%%%%%%%%%%%%%%%%%%%%%%%%%%%

Consider the general scheme presented in \cref{alg-gen-sch}. This scheme works
with an arbitrary oracle $\Map{\GCal}{\VectorSpace{E}}{\VectorSpace{E}\Dual}$
satisfying the following condition:
\begin{equation}\label{oracle-asum}
  \exists x^* \in \Ball{x_0}{R} \colon \quad
  \DualPairing{\GCal(x)}{x - x^*} \geq 0,
  \quad \forall x \in \VectorSpace{E}.
\end{equation}
The point $x^*$ from \cref{oracle-asum} is typically called a \emph{solution} of
our problem. For the general problem with convex structure, represented by the
First-Order Oracle $g$ and the Separation Oracle $g_Q$ for the solid $Q$, the
oracle $\GCal$ is usually defined as follows: $\GCal(x) \DefinedEqual g(x)$, if
$x \in \Interior Q$, and $\GCal(x) \DefinedEqual g_Q(x)$, otherwise. To ensure
\cref{oracle-asum}, the constant $R$ needs to be chosen sufficiently big so that
$Q \subseteq \Ball{x_0}{R}$.

\begin{SimpleAlgorithm}[%
  title={General Scheme of Subgradient Ellipsoid Method},%
  label=alg-gen-sch,%
  width=0.85\linewidth%
]%
  \begin{AlgorithmGroup}[Input]
    Point $x_0 \in \VectorSpace{E}$ and scalar $R > 0$.
  \end{AlgorithmGroup}
  \begin{AlgorithmGroup}[Initialization]
    Define the functions $\ell_0(x) \DefinedEqual 0$, %
    $\omega_0(x) \DefinedEqual \frac{1}{2} \Norm{x - x_0}^2$.
  \end{AlgorithmGroup}
  \AlgorithmGroupSeparator

  \begin{AlgorithmGroup}[For $k \geq 0$ iterate]
    \begin{AlgorithmSteps}
      \AlgorithmStep
        Query the oracle to obtain $g_k \DefinedEqual \GCal(x_k)$.

      \AlgorithmStep\label{alg:NewEllipsoidMethod:ComputeU}
        Compute $U_k \DefinedEqual \max_{x \in \Omega_k \cap L_k^-}
        \DualPairing{g_k}{x_k - x}$, where
        \[
          \Omega_k \DefinedEqual \SetBuilder{x \in \VectorSpace{E}}{
            \omega_k(x) \leq \tfrac{1}{2} R^2},
          \qquad
          L_k^- \DefinedEqual \SetBuilder{x \in \VectorSpace{E}}{
            \ell_k(x) \leq 0}.
        \]

      \AlgorithmStep%
        Choose some coefficients $a_k, b_k \geq 0$ and update the functions
        \begin{equation}\label{upd-ell-omega}
          {\begin{aligned}
            \ell_{k+1}(x) &\DefinedEqual \ell_k(x)
              + a_k \DualPairing{g_k}{x - x_k}, \\
            \omega_{k+1}(x) &\DefinedEqual \omega_k(x)
              + \tfrac{1}{2} b_k (U_k - \DualPairing{g_k}{x_k - x})
                \DualPairing{g_k}{x - x_k}.
          \end{aligned}}
        \end{equation}

      \AlgorithmStep
        Set $x_{k+1} \DefinedEqual \argmin_{x \in \VectorSpace{E}}
          [\ell_{k+1}(x) + \omega_{k+1}(x)]$.
    \end{AlgorithmSteps}
  \end{AlgorithmGroup}
\end{SimpleAlgorithm}

Note that, in \cref{alg-gen-sch}, $\omega_k$ are strictly convex quadratic
functions and $\ell_k$ are affine functions. Therefore, the sets $\Omega_k$ are
certain ellipsoids and $L_k^-$ are certain halfspaces (possibly degenerate).

Let us show that \cref{alg-gen-sch} is a cutting-plane scheme in which the sets
$\Omega_k \cap L_k^-$ are the localizers of the solution $x^*$.

\begin{lemma}\label{lm-set-inc}%
  In \cref{alg-gen-sch}, for all $k \geq 0$, we have $x^* \in \Omega_k \cap
  L_k^-$ and $\hat{Q}_{k+1} \subseteq \Omega_{k+1} \cap L_{k+1}^-$, where
  $\hat{Q}_{k+1} \DefinedEqual \SetBuilder{x \in \Omega_k \cap
  L_k^-}{\DualPairing{g_k}{x - x_k} \leq 0}$.
\end{lemma}

\begin{proof}
  Let us prove the claim by induction. Clearly, $\Omega_0 = \Ball{x_0}{R}$ and
  $L_0^- = \VectorSpace{E}$, hence $\Omega_0 \cap L_0^- = \Ball{x_0}{R} \ni x^*$
  by \cref{oracle-asum}. Suppose we have already proved that $x^* \in \Omega_k
  \cap L_k^-$ for some $k \geq 0$. Combining this with \cref{oracle-asum}, we
  obtain $x^* \in \hat{Q}_{k+1}$, so it remains to show that $\hat{Q}_{k+1}
  \subseteq \Omega_{k+1} \cap L_{k+1}^-$. Let $x \in \hat{Q}_{k+1}$ ($\subseteq
  \Omega_k \cap L_k^-$) be arbitrary. Note that $0 \leq \DualPairing{g_k}{x_k -
  x} \leq U_k$. Hence, by \cref{upd-ell-omega}, $\ell_{k+1}(x) \leq \ell_k(x)
  \leq 0$ and $\omega_{k+1}(x) \leq \omega_k(x) \leq \frac{1}{2} R^2$, which
  means that $x \in \Omega_{k+1} \cap L_{k+1}^-$.
\end{proof}

Next, let us establish an important representation of the ellipsoids~$\Omega_k$
via the functions~$\ell_k$ and the test points~$x_k$. For this, let us define
$G_k \DefinedEqual \Hessian \omega_k(0)$ for each $k \geq 0$. Observe that these
operators satisfy the following simple relations (cf. \cref{upd-ell-omega}):
\begin{equation}\label{G-next}
  G_0 = B,
  \qquad
  G_{k+1} = G_k + b_k g_k g_k\Adjoint,
  \quad
  k \geq 0.
\end{equation}
Also, let us define the sequence $R_k > 0$ by the recurrence
\begin{equation}\label{R-rec}
  R_0 = R,
  \qquad
  R_{k+1}^2 = R_k^2 + ( a_k + \tfrac{1}{2} b_k U_k )^2
    \frac{\RelativeDualNorm{g_k}{G_k}^2}{1 + b_k \RelativeDualNorm{g_k}{G_k}^2},
  \quad k \geq 0.
\end{equation}

\begin{lemma}\label{lm-alg-main-ineq}%
  In \cref{alg-gen-sch}, for all $k \geq 0$, we have
  \[
    \Omega_k = \SetBuilder{x \in \VectorSpace{E}}{
      -\ell_k(x) + \tfrac{1}{2} \RelativeNorm{x - x_k}{G_k}^2
      \leq \tfrac{1}{2} R_k^2
    }.
  \]
  In particular, for all $k \geq 0$ and all $x \in \Omega_k \cap L_k^-$, we have
  $\RelativeNorm{x - x_k}{G_k} \leq R_k$.
\end{lemma}

\begin{proof}
  Let $\Map{\psi_k}{\VectorSpace{E}}{\RealField}$ be the function $\psi_k(x)
  \DefinedEqual \ell_k(x) + \omega_k(x)$. Note that $\psi_k$ is a quadratic
  function with Hessian $G_k$ and minimizer $x_k$. Hence, for any $x \in
  \VectorSpace{E}$, we have
  \begin{equation}\label{psi-canon}
    \psi_k(x) = \psi_k^* + \tfrac{1}{2} \RelativeNorm{x - x_k}{G_k}^2,
  \end{equation}
  where $\psi_k^* \DefinedEqual \min_{x \in \VectorSpace{E}} \psi_k(x)$.

  Let us compute $\psi_k^*$. Combining \cref{upd-ell-omega,psi-canon,G-next},
  for any $x \in \VectorSpace{E}$, we obtain
  \begin{equation}\label{psi-next}
    \begin{aligned}
      \psi_{k+1}(x)
      &= \psi_k(x)
        + (a_k + \tfrac{1}{2} b_k U_k) \DualPairing{g_k}{x - x_k}
        + \tfrac{1}{2} b_k \DualPairing{g_k}{x - x_k}^2 \\
      &= \psi_k^* + \tfrac{1}{2} \RelativeNorm{x - x_k}{G_k}^2
        + (a_k + \tfrac{1}{2} b_k U_k) \DualPairing{g_k}{x - x_k}
        + \tfrac{1}{2} b_k \DualPairing{g_k}{x - x_k }^2 \\
      &= \psi_k^* + \tfrac{1}{2} \RelativeNorm{x - x_k}{G_{k+1}}^2
        + (a_k + \tfrac{1}{2} b_k U_k) \DualPairing{g_k}{x - x_k},
    \end{aligned}
  \end{equation}
  Therefore,
  \begin{equation}\label{min-psi-next}
    \psi_{k+1}^*
    = \psi_k^* - \tfrac{1}{2} ( a_k + \tfrac{1}{2} b_k U_k )^2
      \RelativeDualNorm{g_k}{G_{k+1}}^2
    = \psi_k^*
    - \tfrac{1}{2} ( a_k + \tfrac{1}{2} b_k U_k )^2
      \frac{\RelativeDualNorm{g_k}{G_k}^2}{
        1 + b_k \RelativeDualNorm{g_k}{G_k}^2}.
  \end{equation}
  where the last identity follows from the fact that $G_{k+1}^{-1} g_k =
  G_k^{-1} g_k / (1 + b_k \RelativeDualNorm{g_k}{G_k}^2)$ (since $G_{k+1}
  G_k^{-1} g_k = (1 + b_k \RelativeDualNorm{g_k}{G_k}^2) g_k$ in view of
  \cref{G-next}). Since \cref{min-psi-next} is true for any $k \geq 0$ and since
  $\psi_0^* = 0$, we thus obtain, in view of \cref{R-rec}, that
  \begin{equation}\label{min-psi-via-R}
    \psi_k^* = \tfrac{1}{2} (R^2 - R_k^2).
  \end{equation}

  Let $x \in \Omega_k$ be arbitrary. Using the definition of $\psi_k(x)$ and
  \cref{min-psi-via-R}, we obtain
  \[
    -\ell_k(x) + \tfrac{1}{2} \RelativeNorm{x - x_k}{G_k}^2
    = \omega_k(x) - \psi_k^*
    = \omega_k(x) + \tfrac{1}{2} (R_k^2 - R^2).
  \]
  Thus, $x \in \Omega_k \iff \omega_k(x) \leq \frac{1}{2} R^2 \iff -\ell_k(x) +
  \frac{1}{2} \RelativeNorm{x - x_k}{G_k}^2 \leq \frac{1}{2} R_k^2$. In
  particular, for any $x \in \Omega_k \cap L_k^-$, we have $\ell_k(x) \leq 0$
  and hence $\RelativeNorm{x - x_k}{G_k} \leq R_k$.
\end{proof}

\Cref{lm-alg-main-ineq} has several consequences. First, we see that the
localizers $\Omega_k \cap L_k^-$ are contained in the ellipsoids
$\SetBuilder{x}{\RelativeNorm{x - x_k}{G_k} \leq R_k}$ whose centers are the
test points $x_k$.

Second, we get a uniform upper bound on the function $-\ell_k$ on the ellipsoid
$\Omega_k$: $-\ell_k(x) \leq \frac{1}{2} R_k^2$ for all $x \in \Omega_k$. This
observation leads us to the following definition of the \emph{sliding gap}:
\begin{equation}\label{def-sli-gap}
  \Delta_k \DefinedEqual
  \max_{x \in \Omega_k} \frac{1}{\Gamma_k} [-\ell_k(x)]
  = \max_{x \in \Omega_k} \frac{1}{\Gamma_k} \sum_{i=0}^{k-1}
    a_i \DualPairing{g_i}{x_i - x},
  \quad
  k \geq 1,
\end{equation}
provided that $\Gamma_k \DefinedEqual \sum_{i=0}^{k-1} a_i \DualNorm{g_i} > 0$.
According to our observation, we have
\begin{equation}\label{rate-sli-gap-bas}
  \Delta_k \leq \frac{R_k^2}{2 \Gamma_k}.
\end{equation}
At the same time, $\Delta_k \geq 0$ in view of \cref{lm-set-inc,oracle-asum}

Comparing the definition~\eqref{def-sli-gap} of the sliding gap $\Delta_k$ with
the definition~\eqref{def-gap} of the gap $\delta_k(a\UpperIndex{k})$ for the
semicertificate $a\UpperIndex{k} \DefinedEqual (a_0, \dots, a_{k-1})$, we see
that they are almost identical. The only difference between them is that the
solid $\Omega_k$, over which the maximum is taken in the definition of the
sliding gap, depends on the iteration counter $k$. This seems to be unfortunate
because we cannot guarantee that \emph{each} $\Omega_k$ contains the feasible
set~$Q$ (as required in the definition of gap) even if so does the initial solid
$\Omega_0 = \Ball{x_0}{R}$. However, this problem can be dealt with. Namely, in
\cref{sec-gen-acc-cert}, we will show that the semicertificate $a\UpperIndex{k}$
can be efficiently converted into another semicertificate
$\lambda\UpperIndex{k}$ for which $\delta_k(\lambda\UpperIndex{k}) \leq
\Delta_k$ when taken over the initial solid $\Omega \DefinedEqual \Omega_0$.
Thus, the sliding gap $\Delta_k$ is a meaningful measure of convergence rate of
\cref{alg-gen-sch} and it makes sense to call the coefficients $a\UpperIndex{k}$
a \emph{preliminary semicertificate}.

Let us now demonstrate that, for a suitable choice of the coefficients $a_k$ and
$b_k$ in \cref{alg-gen-sch}, we can ensure that the sliding gap $\Delta_k$
converges to zero.

\begin{remark}\label{as-g-not-0}%
  From now on, in order to avoid taking into account some trivial degenerate
  cases, it will be convenient to make the following minor technical assumption:
  \[
    \text{In \cref{alg-gen-sch}, $g_k \neq 0$ for all $k \geq 0$}.
  \]
  Indeed, when the oracle reports $g_k = 0$ for some $k \geq 0$, it usually
  means that the test point $x_k$, at which the oracle was queried, is, in fact,
  an exact solution to our problem. For example, if the standard oracle for a
  problem with convex structure has reported $g_k = 0$, we can terminate the
  method and return the certificate $\lambda \DefinedEqual (0, \dots, 0, 1)$ for
  which the residual $\epsilon_k(\lambda) = 0$.
\end{remark}

Let us choose the coefficients $a_k$ and $b_k$ in the following way:
\begin{equation}\label{def-ab}
  a_k \DefinedEqual
    \frac{\alpha_k R + \frac{1}{2} \theta \gamma R_k}{
      \RelativeDualNorm{g_k}{G_k}},
  \qquad
  b_k \DefinedEqual
    \frac{\gamma}{\RelativeDualNorm{g_k}{G_k}^2},
  \qquad
  k \geq 0,
\end{equation}
where $\alpha_k, \theta, \gamma \geq 0$ are certain coefficients to be chosen
later.

According to \cref{rate-sli-gap-bas}, to estimate the convergence rate of the
sliding gap, we need to estimate the rate of growth of the coefficients $R_k$
and $\Gamma_k$ from above and below, respectively. Let us do this.

\begin{lemma}\label{lm-Rk-ubd}%
  In \cref{alg-gen-sch} with parameters~\eqref{def-ab}, for all $k \geq 0$, we
  have
  \begin{equation}\label{Rk-ubd-0}
    R_k^2 \leq [q_c(\gamma)]^k C_k R^2,
  \end{equation}
  where $q_c(\gamma) \DefinedEqual 1 + \frac{c \gamma^2}{2 (1 + \gamma)}$, $c
  \DefinedEqual \tfrac{1}{2} (\tau + 1) (\theta + 1)^2$, $C_k \DefinedEqual 1 +
  \frac{\tau + 1}{\tau} \sum_{i=0}^{k-1} \alpha_i^2$ and $\tau > 0$ can be
  chosen arbitrarily. Moreover, if $\alpha_k = 0$ for all $k \geq 0$, then,
  $R_k^2 = [q_c(\gamma)]^k R^2$ for all $k \geq 0$ with $c \DefinedEqual
  \frac{1}{2} (\theta + 1)^2$.
\end{lemma}

\begin{proof}
  By the definition of $U_k$ and \cref{lm-alg-main-ineq}, we have
  \begin{equation}\label{U-ubd}
    U_k
    = \max_{x \in \Omega_k \cap L_k^-} \DualPairing{g_k}{x_k - x}
    \leq \max_{\RelativeNorm{x - x_k}{G_k} \leq R_k}
      \DualPairing{g_k}{x_k - x}
    = R_k \RelativeDualNorm{g_k}{G_k}.
  \end{equation}
  At the same time, $U_k \geq 0$ in view of \cref{lm-set-inc,oracle-asum}.
  Hence,
  \[
    \begin{aligned}
      (a_k + \tfrac{1}{2} b_k U_k)^2
        \frac{\RelativeDualNorm{g_k}{G_k}^2}{
          1 + b_k \RelativeDualNorm{g_k}{G_k}^2}
      &\leq (a_k + \tfrac{1}{2} b_k R_k \RelativeDualNorm{g_k}{G_k})^2
          \frac{\RelativeDualNorm{g_k}{G_k}^2}{
            1 + b_k \RelativeDualNorm{g_k}{G_k}^2} \\
      &= \frac{1}{1 + \gamma}
        \bigl( \alpha_k R + \tfrac{1}{2} (\theta + 1) \gamma R_k \bigr)^2,
    \end{aligned}
  \]
  where the identity follows from \cref{def-ab}. Combining this with
  \cref{R-rec}, we obtain
  \begin{equation}\label{R-ubd-bas}
    R_{k+1}^2
    \leq R_k^2 + \frac{1}{1 + \gamma}
      \bigl( \alpha_k R + \tfrac{1}{2} (\theta + 1) \gamma R_k \bigr)^2.
  \end{equation}

  Note that, for any $\xi_1, \xi_2 \geq 0$ and any $\tau > 0$, we have
  \[
    (\xi_1 + \xi_2)^2
    = \xi_1^2 + 2 \xi_1 \xi_2 + \xi_2^2
    \leq \frac{\tau + 1}{\tau} \xi_1^2 + (\tau + 1) \xi_2^2
    = (\tau + 1) \Bigl( \frac{1}{\tau} \xi_1^2 + \xi_2^2 \Bigr)
  \] 
  (look at the minimum of the right-hand side in $\tau$). Therefore, for
  arbitrary $\tau > 0$,
  \[
    R_{k+1}^2
    \leq R_k^2 + \frac{\tau + 1}{1 + \gamma} \Bigl(
      \frac{1}{\tau} \alpha_k^2 R^2
      + \tfrac{1}{4} (\theta + 1)^2 \gamma^2 R_k^2 \Bigr)
    = q R_k^2 + \beta_k R^2,
  \]
  where we denote $q \DefinedEqual q_c(\gamma) \geq 1$ and $\beta_k
  \DefinedEqual \frac{\tau + 1}{\tau (1 + \gamma)} \alpha_k^2$. Dividing both
  sides by $q^{k+1}$, we get
  \[
    \frac{R_{k+1}^2}{q^{k+1}}
    \leq \frac{R_k^2}{q^k} + \frac{\beta_k R^2}{q^{k+1}}.
  \]
  Since this is true for any $k \geq 0$, we thus obtain, in view of
  \cref{R-rec}, that
  \[
    \frac{R_k^2}{q^k}
    \leq \frac{R_0^2}{q^0} + R^2 \sum_{i=0}^{k-1} \frac{\beta_i}{q^{i+1}}
    = \biggl( 1 + \sum_{i=0}^{k-1} \frac{\beta_i}{q^{i+1}} \biggr) R^2,
  \]
  Multiplying both sides by $q^k$ and using that $\frac{\beta_i}{q^{i+1}} \leq
  \frac{\tau + 1}{\tau} \alpha_i^2$, we come to \cref{Rk-ubd-0}.

  When $\alpha_k = 0$ for all $k \geq 0$, we have $\ell_k = 0$ and $L_k^- =
  \VectorSpace{E}$ for all $k \geq 0$. Therefore, by \cref{lm-alg-main-ineq},
  $\Omega_k = \SetBuilder{x}{\RelativeNorm{x - x_k}{G_k} \leq R_k}$ and hence
  \cref{U-ubd} is, in fact, an equality. Consequently, \cref{R-ubd-bas} becomes
  $R_{k+1}^2 = R_k^2 + \frac{c \gamma^2}{2 (1 + \gamma)} R_k^2 = q_c(\gamma)
  R_k^2$, where $c \DefinedEqual \frac{1}{2} (\theta + 1)^2$.
\end{proof}

\begin{remark}%
  From the proof, one can see that the quantity $C_k$ in \cref{lm-Rk-ubd} can be
  improved up to $C_k' \DefinedEqual 1 + \frac{\tau + 1}{\tau (1 + \gamma)}
  \sum_{i=0}^{k-1} \frac{\alpha_i^2}{[q_c(\gamma)]^{i+1}}$.
\end{remark}

\begin{lemma}\label{lm-Gamma-lbd}%
  In \cref{alg-gen-sch} with parameters~\eqref{def-ab}, for all $k \geq 1$, we
  have
  \begin{equation}\label{Gamma-lbd}
    \Gamma_k \geq R \Bigl( \, \sum_{i=0}^{k-1} \alpha_i
      + \tfrac{1}{2} \theta
        \sqrt{\gamma n \bigl[ (1 + \gamma)^{k/n} - 1 \bigr]} \, \Bigr).
  \end{equation}
\end{lemma}

\begin{proof}
  By the definition of $\Gamma_k$ and \cref{def-ab}, we have
  \[
    \Gamma_k
    = \sum_{i=0}^{k-1} a_i \DualNorm{g_i}
    = R \sum_{i=0}^{k-1} \alpha_i \rho_i
      + \tfrac{1}{2} \theta \gamma \sum_{i=0}^{k-1} R_i \rho_i,
  \]
  where $\rho_i \DefinedEqual \DualNorm{g_i} / \RelativeDualNorm{g_i}{G_i}$. Let
  us estimate each sum from below separately.

  For the first sum, we can use the trivial bound $\rho_i \geq 1$, which is
  valid for any $i \geq 0$ (since $G_i \succeq B$ in view of \cref{G-next}).
  This gives us $\sum_{i=0}^{k-1} \alpha_i \rho_i \geq \sum_{i=0}^{k-1}
  \alpha_i$.

  Let us estimate the second sum. According to \cref{R-rec}, for any $i \geq 0$,
  we have $R_i \geq R$. Hence, $\sum_{i=0}^{k-1} R_i \rho_i \geq R
  \sum_{i=0}^{k-1} \rho_i \geq R ( \sum_{i=0}^{k-1} \rho_i^2 )^{1/2}$ and it
  remains to lower bound $\sum_{i=0}^{k-1} \rho_i^2$. By \cref{G-next,def-ab},
  $G_0 = B$ and $G_{i+1} = G_i + \gamma g_i g_i\Adjoint /
  \RelativeDualNorm{g_i}{G_i}^2$ for all $i \geq 0$. Therefore,
  \[
    \begin{aligned}
      \sum_{i=0}^{k-1} \rho_i^2
      &= \frac{1}{\gamma} \sum_{i=0}^{k-1} ( \Trace G_{i+1} - \Trace G_i )
      = \frac{1}{\gamma} (\Trace G_k - \Trace B)
      = \frac{1}{\gamma} (\Trace G_k - n) \\
      &\geq \frac{n}{\gamma} \bigl[ (\det G_k)^{1/n} - 1 \bigr]
      = \frac{n}{\gamma} \bigl[ (1 + \gamma)^{k/n} - 1 \bigr],
    \end{aligned}
  \]
  where we have applied the arithmetic-geometric mean inequality. Combining the
  obtained estimates, we get \cref{Gamma-lbd}.
\end{proof}

%%%%%%%%%%%%%%%%%%%%%%%%%%%%%%%%%%%%%%%%%%%%%%%%%%%%%%%%%%%%%%%%%%%%%%%%%%%%%%%%
%%%%%%%%%%%%%%%%%%%%%%%%%%%%%%%%%%%%%%%%%%%%%%%%%%%%%%%%%%%%%%%%%%%%%%%%%%%%%%%%
\section{Main Instances of General Scheme}\label{sec-main-inst}
%%%%%%%%%%%%%%%%%%%%%%%%%%%%%%%%%%%%%%%%%%%%%%%%%%%%%%%%%%%%%%%%%%%%%%%%%%%%%%%%
%%%%%%%%%%%%%%%%%%%%%%%%%%%%%%%%%%%%%%%%%%%%%%%%%%%%%%%%%%%%%%%%%%%%%%%%%%%%%%%%

Let us now consider several possibilities for choosing the coefficients
$\alpha_k$, $\theta$ and $\gamma$ in \cref{def-ab}.

\subsection{Subgradient Method}\label{sec-subgrad-met}

The simplest possibility is to choose
\[
  \alpha_k > 0,
  \qquad
  \theta \DefinedEqual 0,
  \qquad
  \gamma \DefinedEqual 0.
\]
In this case, $b_k = 0$ for all $k \geq 0$, so $G_k = B$ and $\omega_k(x) =
\omega_0(x) = \frac{1}{2} \Norm{x - x_0}^2$ for all $x \in \VectorSpace{E}$ and
all $k \geq 0$ (see \cref{G-next,upd-ell-omega}). Consequently, the new test
points $x_{k+1}$ in \cref{alg-gen-sch} are generated according to the following
rule: 
\[
  x_{k+1} = \argmin_{x \in \VectorSpace{E}} \Bigl[ \,
    \sum_{i=0}^k a_i \DualPairing{g_i}{x - x_i}
    + \tfrac{1}{2} \Norm{x - x_0}^2 \Bigr],
\]
where $a_i = \alpha_i R / \DualNorm{g_i}$. Thus, \cref{alg-gen-sch} is the
Subgradient Method: $x_{k+1} = x_k - a_k g_k$.

In this example, each ellipsoid $\Omega_k$ is simply a ball: $\Omega_k =
\Ball{x_0}{R}$ for all $k \geq 0$. Hence, the sliding gap $\Delta_k$, defined in
\cref{def-sli-gap}, does not ``slide'' and coincides with the gap of the
semicertificate $a \DefinedEqual (a_0, \dots, a_{k-1})$ on the solid
$\Ball{x_0}{R}$:
\[
  \Delta_k = \max_{x \in \Ball{x_0}{R}}
    \frac{1}{\Gamma_k} \sum_{i=0}^{k-1} a_i \DualPairing{g_i}{x_i - x}.
\]

In view of \cref{lm-Rk-ubd,lm-Gamma-lbd}, for all $k \geq 1$, we have
\[
  R_k^2 \leq \Bigl(1 + \sum_{i=0}^{k-1} \alpha_i^2 \Bigr) R^2,
  \qquad
  \Gamma_k \geq R \sum_{i=0}^{k-1} \alpha_i
\]
(tend $\tau \to +\infty$ in \cref{lm-Rk-ubd}). Substituting these estimates into
\cref{rate-sli-gap-bas}, we obtain the following well-known estimate for the gap
in the Subgradient Method:
\[
  \Delta_k \leq \frac{1 + \sum_{i=0}^{k-1} \alpha_i^2}{
    2 \sum_{i=0}^{k-1} \alpha_i} R.
\]

The standard strategies for choosing the coefficients $\alpha_i$ are as follows
(see, e.g., \crefName{section}~3.2.3 in~\cite{Nesterov2018Book}):
\begin{enumerate}
  \item We fix in advance the number of iterations $k \geq 1$ of the method and
  use \emph{constant} coefficients $\alpha_i \DefinedEqual \frac{1}{\sqrt{k}}$,
  $0 \leq i \leq k - 1$. This corresponds to the so-called \emph{Short-Step}
  Subgradient Method, for which we have
  \[
    \Delta_k \leq \frac{R}{\sqrt{k}}.
  \]

  \item Alternatively,  we can use \emph{time-varying} coefficients $\alpha_i
  \DefinedEqual \frac{1}{\sqrt{i + 1}}$, $i \geq 0$. This approach does not
  require us to fix in advance the number of iterations $k$. However, the
  corresponding convergence rate estimate becomes slightly worse:
  \[
    \Delta_k \leq \frac{\ln k + 2}{2 \sqrt{k}} R.
  \]
  (Indeed, $\sum_{i=0}^{k-1} \alpha_i^2 = \sum_{i=1}^k \frac{1}{i} \leq \ln k +
  1$, while $\sum_{i=0}^{k-1} \alpha_i \geq \sqrt{k}$.)
\end{enumerate}

\begin{remark}\label{rm-subgrad-met-wo-log-fact}%
  If we allow projections onto the feasible set, then, for the resulting
  Subgradient Method with time-varying coefficients $\alpha_i$, one can
  establish the $\BigO(1/\sqrt{k})$ convergence rate for the ``truncated'' gap
  \[
    \Delta_{k_0, k}
    \DefinedEqual \max_{x \in \Ball{x_0}{R}} \frac{1}{\Gamma_{k_0, k}}
      \sum_{i=k_0}^{k} a_i \DualPairing{g_i}{x_i - x},
  \]
  where $\Gamma_{k_0, k} \DefinedEqual \sum_{i=k_0}^k a_i \DualNorm{g_i}$,
  $k_0 \DefinedEqual \Ceil{k/2}$. For more details, see \crefName{section}~5.2.1
  in~\cite{BenTalNemirovski2021BookModernCO} or \crefName{section}~3.1.1
  in~\cite{Lan2020Book}.
\end{remark}

\subsection{Standard Ellipsoid Method}\label{sec-stand-ell-met}

Another extreme choice is the following one:
\begin{equation}\label{alpha-theta-st-ell}
  \alpha_k \DefinedEqual 0,
  \qquad
  \theta \DefinedEqual 0,
  \qquad
  \gamma > 0.
\end{equation}
For this choice, we have $a_k = 0$ for all $k \geq 0$. Hence, $\ell_k = 0$ and
$L_k^- = \VectorSpace{E}$ for all $k \geq 0$. Therefore, the localizers in this
method are the following ellipsoids (see \cref{lm-alg-main-ineq}):
\begin{equation}\label{loc-st-ell}
  \Omega_k \cap L_k^- = \Omega_k
  = \SetBuilder{x \in \VectorSpace{E}}{\RelativeNorm{x - x_k}{G_k} \leq R_k},
  \qquad
  k \geq 0.
\end{equation}

Observe that, in this example, $\Gamma_k \equiv \sum_{i=0}^{k-1} a_i
\DualNorm{g_i} = 0$ for all $k \geq 1$, so there is no preliminary
semicertificate and the sliding gap is undefined. However, we can still ensure
the convergence to zero of a certain meaningful measure of optimality, namely,
the \emph{average radius} of the localizers~$\Omega_k$:
\begin{equation}\label{def-avrad}
  \AverageRadius \Omega_k \DefinedEqual (\Volume \Omega_k)^{1/n},
  \qquad
  k \geq 0.
\end{equation}
Indeed, let us define the following functions for any real $c, p > 0$:
\begin{equation}\label{def-zeta-q}
  q_c(\gamma) \DefinedEqual 1 + \frac{c \gamma^2}{2 (1 + \gamma)},
  \qquad
  \zeta_{p, c}(\gamma) \DefinedEqual \frac{[q_c(\gamma)]^p}{1 + \gamma},
  \qquad
  \gamma > 0.
\end{equation}
According to \cref{lm-Rk-ubd}, for any $k \geq 0$, we have
\begin{equation}\label{R-sqr-st-ell}
  R_k^2 = [q_{1/2}(\gamma)]^k R^2.
\end{equation}
At the same time, in view of \cref{G-next,def-ab}, $\det G_k = \prod_{i=0}^{k-1}
(1 + b_i \RelativeDualNorm{g_i}{G_i}^2) = (1 + \gamma)^k$ for all $k \geq 0$.
Combining this with \cref{loc-st-ell,def-avrad,def-zeta-q}, we obtain, for any
$k \geq 0$, that
\begin{equation}\label{std-ell-avrad-prel}
  \AverageRadius \Omega_k
  = \frac{R_k}{(\det G_k)^{1 / (2 n)}}
  = \frac{[q_{1/2}(\gamma)]^{k/2} R}{(1 + \gamma)^{k / (2 n)}}
  = [\zeta_{n, 1/2}(\gamma)]^{k / (2 n)} R.
\end{equation}

Let us now choose $\gamma$ which minimizes $\AverageRadius \Omega_k$. For such
computations, the following auxiliary result is useful (see
\cref{sec-proof-lm-min-zeta} for the proof).

\begin{lemma}\label{lm-min-zeta}%
  For any $c \geq 1/2$ and any $p \geq 2$, the function $\zeta_{p, c}$, defined
  in \cref{def-zeta-q}, attains its minimum at a unique point
  \begin{equation}\label{gamma-opt}
    \gamma_c(p) \DefinedEqual
      \frac{2}{\sqrt{c^2 p^2 - (2 c - 1)} + c p - 1}
    \in \ClosedClosedInterval[\bigg]{\frac{1}{c p}}{\frac{2}{cp}}
  \end{equation}
  with the corresponding value $\zeta_{p, c}\bigl( \gamma_c(p) \bigr) \leq e^{-1
  / (2 c p)}$.
\end{lemma}

Applying \cref{lm-min-zeta} to \cref{std-ell-avrad-prel}, we see that the
optimal value of $\gamma$ is
\begin{equation}\label{gamma-st-ell}
  \gamma \DefinedEqual \gamma_{1/2}(n)
  = \frac{2}{n/2 + n/2 - 1} = \frac{2}{n-1},
\end{equation}
for which $\zeta_{n, 1/2}(\gamma) \leq e^{-1/n}$. With this choice of $\gamma$,
we obtain, for all $k \geq 0$, that
\begin{equation}\label{std-ell-avrad-final}
  \AverageRadius \Omega_k \leq e^{-k / (2 n^2)} R.
\end{equation}

One can check that \cref{alg-gen-sch} with parameters~\eqref{def-ab},
\eqref{alpha-theta-st-ell}, \eqref{gamma-st-ell} is, in fact, the standard
Ellipsoid Method (see \cref{rm-st-ell-met}).

\subsection{Ellipsoid Method with Preliminary Semicertificate}
\label{sec-ell-met-w-prel-cert}

As we have seen, we cannot measure the convergence rate of the standard
Ellipsoid Method using the sliding gap because there is no preliminary
semicertificate in this method. Let us present a modification of the standard
Ellipsoid Method which does not have this drawback but still enjoys the same
convergence rate as the original method (up to some absolute constants).

For this, let us choose the coefficients in the following way:
\begin{equation}\label{alpha-theta-ell-w-prel-cert}
  \alpha_k \DefinedEqual 0,
  \qquad
  \theta \DefinedEqual \sqrt{2} - 1 \ (\approx 0.41),
  \qquad
  \gamma > 0.
\end{equation}
Then, in view of \cref{lm-Rk-ubd}, for all $k \geq 0$, we have
\begin{equation}\label{ell-prel-cert-R-ubd}
  R_k^2 = [q_1(\gamma)]^k R^2,
\end{equation}
Also, by \cref{lm-Gamma-lbd}, $\Gamma_k \geq \frac{1}{2} \theta R \sqrt{\gamma n
[ (1 + \gamma)^{k/n} - 1 ]}$ for all $k \geq 1$. Thus, for each $k \geq 1$, we
obtain the following estimate for the sliding gap (see \cref{rate-sli-gap-bas}):
\begin{equation}\label{ell-prel-cert-Delta-bas}
  \Delta_k
  \leq \frac{[q_1(\gamma)]^k R}{
    \theta \sqrt{\gamma n [ (1 + \gamma)^{k/n} - 1 ]}}
  = \frac{1}{\theta \kappa_k(\gamma, n)} [\zeta_{2 n, 1}(\gamma)]^{k/(2 n)} R,
\end{equation}
where $\kappa_k(\gamma, n) \DefinedEqual \sqrt{\gamma n ( 1 - \frac{1}{(1 +
\gamma)^{k/n}})}$ and $\zeta_{2 n, 1}(\gamma)$ is defined in \cref{def-zeta-q}.

Note that the main factor in estimate~\eqref{ell-prel-cert-Delta-bas} is
$[\zeta_{2 n, 1}(\gamma)]^{k / (2 n)}$. Let us choose $\gamma$ by minimizing
this expression. Applying \cref{lm-min-zeta}, we obtain
\begin{equation}\label{gamma-mod}
  \gamma
  \DefinedEqual \gamma_1(2 n)
  \in \ClosedClosedInterval[\bigg]{\frac{1}{2 n}}{\frac{1}{n}}.
\end{equation}

\begin{theorem}%
  In \cref{alg-gen-sch} with parameters~\eqref{def-ab},
  \eqref{alpha-theta-ell-w-prel-cert}, \eqref{gamma-mod}, for all $k \geq 1$,
  \[
    \Delta_k \leq 6 e^{-k / (8 n^2)} R.
  \]
\end{theorem}

\begin{proof}
  \ProofPart

  Suppose $k \geq n^2$. According to \cref{lm-min-zeta}, we have $\zeta_{2 n,
  1}(\gamma) \leq e^{-1 / (4 n)}$. Hence, by \cref{ell-prel-cert-Delta-bas},
  $\Delta_k \leq \frac{1}{\theta \kappa_k(\gamma, n)} e^{-k / (8 n^2)} R$. It
  remains to estimate from below $\theta \kappa_k(\gamma, n)$.
  
  Since $k \geq n^2$, we have $(1 + \gamma)^{k/n} \geq (1 + \gamma)^n \geq 1 +
  \gamma n$. Hence, $\kappa_k(\gamma, n) \geq \frac{\gamma n}{\sqrt{1 + \gamma
  n}}$. Note that the function $\tau \mapsto \frac{\tau}{\sqrt{1 + \tau}}$ is
  increasing on $\NonnegativeRay$. Therefore, using
  \cref{gamma-mod}, we obtain $\kappa_k(\gamma, n) \geq \frac{1 / 2}{\sqrt{1 + 1
  / 2}} = \frac{1}{\sqrt{6}}$. Thus, $\theta \kappa_k(\gamma, n) \geq
  \frac{\sqrt{2} - 1}{\sqrt{6}} \geq \frac{1}{6}$ for our choice of $\theta$.

  \ProofPart

  Now suppose $k \leq n^2$. Then, $6 e^{-k / (8 n^2)} \geq 6 e^{-1/8} \geq 5$.
  Therefore, it suffices to prove that $\Delta_k \leq 5 R$ or, in view of
  \cref{def-sli-gap}, that $\DualPairing{g_i}{x_i - x} \leq 5 R \DualNorm{g_i}$,
  where $x \in \Omega_k \cap L_k^-$ and $0 \leq i \leq k-1$ are arbitrary. Note
  that $\DualPairing{g_i}{x_i - x} \leq \RelativeDualNorm{g_i}{G_i}
  \RelativeNorm{x_i - x}{G_i} \leq \DualNorm{g_i} \RelativeNorm{x_i - x}{G_i}$
  since $G_i \succeq B$ (see \cref{G-next}). Hence, it remains to prove that
  $\RelativeNorm{x_i - x}{G_i} \leq 5 R$.

  Recall from \cref{G-next,R-rec} that $G_i \preceq G_k$ and $R_i \leq R_k$.
  Therefore,
  \[
    \begin{aligned}
      \RelativeNorm{x_i - x}{G_i}
      &\leq \RelativeNorm{x_i - x^*}{G_i} + \RelativeNorm{x^* - x}{G_i}
      \leq \RelativeNorm{x_i - x^*}{G_i} + \RelativeNorm{x^* - x}{G_k} \\
      &\leq \RelativeNorm{x_i - x^*}{G_i} + \RelativeNorm{x_k - x^*}{G_k}
        + \RelativeNorm{x_k - x}{G_k}
      \leq R_i + 2 R_k \leq 3 R_k,
    \end{aligned}
  \]
  where the penultimate inequality follows from
  \cref{lm-set-inc,lm-alg-main-ineq}. According to \cref{ell-prel-cert-R-ubd},
  $R_k = [q_1(\gamma)]^{k/2} R \leq [q_1(\gamma)]^{n^2/2} R$ (recall that
  $q_1(\gamma) \geq 1$). Thus, it remains to show that $3 [q_1(\gamma)]^{n^2/2}
  \leq 5$. But this is immediate. Indeed, by \cref{def-zeta-q,gamma-mod}, we
  have $[q_1(\gamma)]^{n^2/2} \leq e^{n^2 \gamma^2 / (4 (1 + \gamma))} \leq
  e^{1/4}$, so $3 [q_1(\gamma)]^{n^2/2} \leq 3 e^{1/4} \leq 5$.
\end{proof}

\subsection{Subgradient Ellipsoid Method}
\label{sec-subgrad-ell-met}

The previous algorithm still shares the drawback of the original Ellipsoid
Method, namely, it does not work when $n \to \infty$. To eliminate this
drawback, let us choose $\alpha_k$ similarly to how this is done in the
Subgradient Method.

Consider the following choice of parameters:
\begin{equation}\label{subgrad-ell-params}
  \alpha_i \DefinedEqual \beta_i \sqrt{\frac{\theta}{\theta + 1}},
  \qquad
  \theta \DefinedEqual \sqrt[3]{2} - 1 \ (\approx 0.26),
  \qquad
  \gamma \DefinedEqual \gamma_1(2 n)
    \in \ClosedClosedInterval[\bigg]{\frac{1}{2 n}}{\frac{1}{n}},
\end{equation}
where $\beta_i > 0$ are certain coefficients (to be specified later) and
$\gamma_1(2 n)$ is defined in \cref{gamma-opt}.

\begin{theorem}%
  In \cref{alg-gen-sch} with parameters~\eqref{def-ab},
  \eqref{subgrad-ell-params}, where $\beta_0 \geq 1$, we have, for all $k \geq
  1$,
  \begin{equation}\label{rate-subg-ell}
    \Delta_k \leq
    \begin{cases}
      \frac{2}{\sum_{i=0}^{k-1} \beta_i}
        ( 1 + \sum_{i=0}^{k-1} \beta_i^2 ) R,
        & \text{if $k \leq n^2$}, \\
      6 e^{-k / (8 n^2)} ( 1 + \sum_{i=0}^{k-1} \beta_i^2 ) R,
        & \text{if $k \geq n^2$}.
    \end{cases}
  \end{equation}
\end{theorem}

\begin{proof}
  Applying \cref{lm-Rk-ubd} with $\tau \DefinedEqual \theta$ and using
  \cref{subgrad-ell-params}, we obtain
  \begin{equation}\label{R-ubd-subg-ell}
    R_k^2 \leq [q_1(\gamma)]^k C_k R^2,
    \qquad
    C_k = 1 + \sum_{i=0}^{k-1} \beta_i^2.
  \end{equation}
  At the same time, by \cref{lm-Gamma-lbd}, we have
  \begin{equation}\label{Gamma-lbd-bas-subg-ell}
    \Gamma_k \geq R \Bigl(
      \sqrt{\frac{\theta}{\theta + 1}} \sum_{i=0}^{k-1} \beta_i
      + \tfrac{1}{2} \theta \sqrt{\gamma n [(1 + \gamma)^{k/n} - 1]}
    \, \Bigr).
  \end{equation}

  Note that $\frac{1}{2} \theta \sqrt{\gamma n} \leq \frac{1}{2} \theta \leq
  \sqrt{\theta / (\theta + 1)}$ by \cref{subgrad-ell-params}. Since $\beta_0
  \geq 1$, we thus obtain
  \begin{equation}\label{Gamma-lbd-1-subg-ell}
    \begin{aligned}
      \Gamma_k
      &\geq \tfrac{1}{2} R \theta \sqrt{\gamma n} \Bigl(
        1 + \sqrt{(1 + \gamma)^{k/n} - 1} \, \Bigr)
      \geq \tfrac{1}{2} R \theta \sqrt{\gamma n} (1 + \gamma)^{k / (2 n)} \\
      &\geq \tfrac{1}{2 \sqrt{2}} R \theta (1 + \gamma)^{k / (2 n)}
      \geq \tfrac{1}{12} R (1 + \gamma)^{k / (2 n)},
    \end{aligned}
  \end{equation}
  where the last two inequalities follow from \cref{subgrad-ell-params}.
  Therefore, by \cref{rate-sli-gap-bas,R-ubd-subg-ell,Gamma-lbd-1-subg-ell},
  \[
    \Delta_k
    \leq \frac{R_k^2}{2 \Gamma_k}
    \leq 6 \frac{[q_1(\gamma)]^k}{(1 + \gamma)^{k/(2n)}} C_k R
    = 6 [\zeta_{2 n, 1}(\gamma)]^{k/(2n)} C_k R,
  \]
  where $\zeta_{2n, 1}(\gamma)$ is defined in \cref{def-zeta-q}. Observe that,
  for our choice of $\gamma$, by \cref{lm-min-zeta}, we have $\zeta_{2 n,
  1}(\gamma) \leq e^{-1 / (4 n)}$. This proves the second estimate\footnote{%
    In fact, we have proved the second estimate in \cref{rate-subg-ell} for all
    $k \geq 1$ (not only for $k \geq n^2$).%
  } %
  in \cref{rate-subg-ell}.

  On the other hand, dropping the second term in \cref{Gamma-lbd-bas-subg-ell},
  we can write
  \begin{equation}\label{Gamma-lbd-2-subg-ell}
    \Gamma_k \geq R \sqrt{\frac{\theta}{\theta + 1}} \sum_{i=0}^{k-1} \beta_i.
  \end{equation}
  Suppose $k \leq n^2$. Then, from \cref{def-zeta-q,subgrad-ell-params}, it
  follows that
  \[
    [q_1(\gamma)]^k
    \leq [q_1(\gamma)]^{n^2}
    \leq e^{\gamma^2 n^2 / (2 (1 + \gamma))}
    \leq \sqrt{e}.
  \]
  Hence, by \cref{R-ubd-subg-ell}, $R_k \leq \sqrt{e} C_k R^2$. Combining this
  with \cref{rate-sli-gap-bas,Gamma-lbd-2-subg-ell}, we obtain
  \[
    \Delta_k
    \leq \frac{1}{2} \sqrt{\frac{e (\theta + 1)}{\theta}}
      \frac{1}{\sum_{i=0}^{k-1} \beta_i} C_k R.
  \]
  By numerical evaluation, one can verify that, for our choice of $\theta$, we
  have $\frac{1}{2} \sqrt{\frac{e (\theta + 1)}{\theta}} \leq 2$. This proves
  the first estimate in \cref{rate-subg-ell}.
\end{proof}

Exactly as in the Subgradient Method, we can use the following two strategies
for choosing the coefficients $\beta_i$:
\begin{enumerate}
  \item We fix in advance the number of iterations $k \geq 1$ of the method and
  use constant coefficients $\beta_i \DefinedEqual \frac{1}{\sqrt{k}}$, $0 \leq
  i \leq k-1$. In this case,
  \begin{equation}\label{subgrad-ell-final-est}
    \Delta_k \leq \begin{cases}
      4 R / \sqrt{k} & \text{if $k \leq n^2$}, \\
      12 R e^{-k / (8 n^2)} & \text{if $k \geq n^2$}.
    \end{cases}
  \end{equation}

  \item We use time-varying coefficients $\beta_i \DefinedEqual \frac{1}{\sqrt{i
  + 1}}$, $i \geq 0$. In this case,
  \[
    \Delta_k \leq \begin{cases}
      2 (\ln k + 2) R / \sqrt{k} & \text{if $k \leq n^2$}, \\
      6 (\ln k + 2) R e^{-k / (8 n^2)} & \text{if $k \geq n^2$}.
    \end{cases}
  \]
\end{enumerate}

Let us discuss convergence rate estimate~\eqref{subgrad-ell-final-est}. Up to
absolute constants, this estimate is exactly the same as in the Subgradient
Method when $k \leq n^2$ and as in the Ellipsoid Method when $k \geq n^2$. In
particular, when $n \to \infty$, we recover the convergence rate of the
Subgradient Method.

To provide a better interpretation of the obtained results, let us compare the
convergence rates of the Subgradient and Ellipsoid methods:
\[
  \begin{aligned}
    \text{Subgradient Method:}& \qquad
      1 / \sqrt{k} \\
    \text{Ellipsoid Method:}& \qquad
      e^{-k / (2 n^2)}.
  \end{aligned}
\]
To compare these rates, let us look at their squared ratio:
\[
  \rho_k \DefinedEqual \Bigl( \frac{1/\sqrt{k}}{e^{-k/(2 n^2)}} \Bigr)^2
  = \frac{e^{k / n^2}}{k}.
\]
Let us find out for which values of $k$ the rate of the Subgradient Method is
better than that of the Ellipsoid Method and vice versa. We assume that $n \geq
2$.

Note that the function $\tau \mapsto e^{\tau} / \tau$ is strictly decreasing on
$\OpenClosedInterval{0}{1}$ and strictly increasing on
$\ClosedOpenInterval{1}{+\infty}$ (indeed, its derivative equals $e^\tau (\tau -
1) / \tau^2$). Hence, $\rho_k$ is strictly decreasing in $k$ for $1 \leq k \leq
n^2$ and strictly increasing in $k$ for $k \geq n^2$. Since $n \geq 2$, we have
$\rho_2 = e^{2 / n^2} / 2 \leq e^{1/2} / 2 \leq 1$. At the same time, $\rho_k
\to +\infty$ when $k \to \infty$. Therefore, there exists a unique integer $K_0
\geq 2$ such that $\rho_k \leq 1$ for all $k \leq K_0$ and $\rho_k \geq 1$ for
all $k \geq K_0$.

Let us estimate $K_0$. Clearly, for any $n^2 \leq k \leq n^2 \ln(2 n)$, we have
\[
  \rho_k \leq \frac{e^{n^2 \ln(2 n) / n^2}}{n^2 \ln(2 n)}
  = \frac{2}{n \ln(2 n)} \leq 1.
\]
while, for any $k \geq 3 n^2 \ln(2 n)$, we have
\[
  \rho_k \geq \frac{e^{3 n^2 \ln(2 n)} / n^2}{3 n^2 \ln(2 n)}
  = \frac{(2 n)^3}{3 n^2 \ln(2 n)}
  = \frac{8 n}{3 \ln(2 n)} \geq 1.
\]
Hence,
\[
  n^2 \ln(2 n) \leq K_0 \leq 3 n^2 \ln(2 n).
\]
Thus, up to an absolute constant, $n^2 \ln(2 n)$ is the switching moment,
starting from which the rate of the Ellipsoid Method becomes better than that of
the Subgradient Method.

Returning to our obtained estimate~\eqref{subgrad-ell-final-est}, we see that,
ignoring absolute constants and ignoring the ``small'' region of the values of
$k$ between $n^2$ and $n^2 \ln n$, our convergence rate is basically the best of
the corresponding convergence rates of the Subgradient and Ellipsoid methods.

%%%%%%%%%%%%%%%%%%%%%%%%%%%%%%%%%%%%%%%%%%%%%%%%%%%%%%%%%%%%%%%%%%%%%%%%%%%%%%%%
%%%%%%%%%%%%%%%%%%%%%%%%%%%%%%%%%%%%%%%%%%%%%%%%%%%%%%%%%%%%%%%%%%%%%%%%%%%%%%%%
\section{Constructing Accuracy Semicertificate}\label{sec-gen-acc-cert}
%%%%%%%%%%%%%%%%%%%%%%%%%%%%%%%%%%%%%%%%%%%%%%%%%%%%%%%%%%%%%%%%%%%%%%%%%%%%%%%%
%%%%%%%%%%%%%%%%%%%%%%%%%%%%%%%%%%%%%%%%%%%%%%%%%%%%%%%%%%%%%%%%%%%%%%%%%%%%%%%%

Let us show how to convert a preliminary accuracy semicertificate, produced by
\cref{alg-gen-sch}, into a semicertificate whose gap on the initial solid is
upper bounded by the sliding gap. The key ingredient here is the following
auxiliary algorithm which was first proposed in
\cite{Nemirovski2010AccCert} for building accuracy certificates in
the standard Ellipsoid Method.

\subsection{Augmentation Algorithm}
\label{sec-augm-alg}

Let $k \geq 0$ be an integer and let $Q_0, \dots, Q_k$ be solids in
$\VectorSpace{E}$ such that
\begin{equation}\label{augm-alg-inc}
  \hat{Q}_i \DefinedEqual \SetBuilder{x \in Q_i}{
    \DualPairing{g_i}{x - x_i} \leq 0} \subseteq Q_{i+1},
  \qquad 0 \leq i \leq k-1,
\end{equation}
where $x_i \in \VectorSpace{E}$, $g_i \in \VectorSpace{E}\Dual$. Further,
suppose that, for any $s \in \VectorSpace{E}\Dual$ and any $0 \leq i \leq k-1$,
we can compute a \emph{dual multiplier} $\mu \geq 0$ such that
\begin{equation}\label{augm-alg-dual-mult}
  \max_{x \in \hat{Q}_i} \DualPairing{s}{x}
  = \max_{x \in Q_i} [
      \DualPairing{s}{x} + \mu \DualPairing{g_i}{x_i - x} ]
\end{equation}
(provided that certain regularity conditions hold). Let us abbreviate any
solution $\mu$ of this problem by $\mu(s, Q_i, x_i, g_i)$.

Consider now the following routine.

\begin{SimpleAlgorithm}[%
  title={Augmentation Algorithm},%
  label=alg-augm,%
  width=0.53\linewidth%
]%
  \begin{AlgorithmGroup}[Input]
    $s_k \in \VectorSpace{E}\Dual$.
  \end{AlgorithmGroup}
  \begin{AlgorithmGroup}[Iterate for $i = k-1, \dots, 0$]
    \begin{AlgorithmSteps}
      \AlgorithmStep
        Compute $\mu_i \DefinedEqual \mu(s_{i+1}, Q_i, x_i, g_i)$.
      \AlgorithmStep
        Set $s_i \DefinedEqual s_{i+1} - \mu_i g_i$.
    \end{AlgorithmSteps}
  \end{AlgorithmGroup}
\end{SimpleAlgorithm}

\begin{lemma}\label{lm-ext-cert}%
  Let $\mu_0, \dots, \mu_{k-1} \geq 0$ be generated by \cref{alg-augm}. Then,
  \[
    \max_{x \in Q_0} \Bigl[
      \DualPairing{s_k}{x}
      + \sum_{i=0}^{k-1} \mu_i \DualPairing{g_i}{x_i - x} \Bigr]
    \leq \max_{x \in Q_k} \DualPairing{s_k}{x}.
  \]
\end{lemma}

\begin{proof}
  Indeed, at every iteration $i = k-1, \dots, 0$, we have
  \[
    \max_{x \in Q_{i+1}} \DualPairing{s_{i+1}}{x}
    \geq \max_{x \in \hat{Q}_i} \DualPairing{s_{i+1}}{x}
    = \max_{x \in Q_i} [ \DualPairing{s_{i+1}}{x}
      + \mu_i \DualPairing{g_i}{x_i - x} ]
    = \max_{x \in Q_i} \DualPairing{s_i}{x}
      + \mu_i \DualPairing{g_i}{x_i}.
  \]
  Summing up these inequalities for $i = 0, \dots, k-1$, we obtain
  \[
    \max_{x \in Q_k} \DualPairing{s_k}{x}
    \geq \max_{x \in Q_0} \DualPairing{s_0}{x}
      + \sum_{i=0}^{k-1} \mu_i \DualPairing{g_i}{x_i}
    = \max_{x \in Q_0} \Bigl[ \DualPairing{s_k}{x}
      + \sum_{i=0}^{k-1} \DualPairing{g_i}{x_i - x} \Bigr],
  \]
  where the identity follows from the fact that $s_0 = s_k - \sum_{i=0}^{k-1}
  \mu_i g_i$.
\end{proof}

\subsection{Methods with Preliminary Certificate}

Let us apply the Augmentation Algorithm for building an accuracy semicertificate
for \cref{alg-gen-sch}. We only consider those instances for which $\Gamma_k
\DefinedEqual \sum_{i=0}^{k-1} a_i \DualNorm{g_i} > 0$ so that the sliding gap
$\Delta_k$ is well-defined:
\[
  \Delta_k
  \DefinedEqual \max_{x \in \Omega_k} \frac{1}{\Gamma_k} [ -\ell_k(x) ]
  = \max_{x \in \Omega_k \cap L_k^-} \frac{1}{\Gamma_k} [ -\ell_k(x) ]
  = \max_{x \in \Omega_k \cap L_k^-}
    \frac{1}{\Gamma_k} \sum_{i=0}^{k-1} a_i \DualPairing{g_i}{x_i - x}.
\]
Recall that the vector $a \DefinedEqual (a_0, \dots, a_{k-1})$ is called a
preliminary semicertificate.

For technical reasons, it will be convenient to add the following termination
criterion into \cref{alg-gen-sch}:
\begin{equation}\label{term-U}
  \text{
    Terminate \cref{alg-gen-sch} at \cref{alg:NewEllipsoidMethod:ComputeU}
    if $U_k \leq \delta \DualNorm{g_k}$,
  }
\end{equation}
where $\delta > 0$ is a fixed constant. Depending on whether this termination
criterion has been satisfied at iteration $k$, we call it a \emph{terminal} or
\emph{nonterminal} iteration, respectively.

Let $k \geq 1$ be an iteration of \cref{alg-gen-sch}. According to
\cref{lm-set-inc}, the sets $Q_i \DefinedEqual \Omega_i \cap L_i^-$ satisfy
\cref{augm-alg-inc}. Since the method has not been terminated during the course
of the previous iterations, we have\footnote{%
  Recall that $g_i \neq 0$ for all $i \geq 0$ by \cref{as-g-not-0}.%
} %
$U_i > 0$ for all $0 \leq i \leq k-1$. Therefore, for any $0 \leq i \leq k-1$,
there exists $x \in Q_i$ such that $\DualPairing{g_i}{x - x_i} < 0$. This
guarantees the existence of dual multiplier in \cref{augm-alg-dual-mult}.

Let us apply \cref{alg-augm} to $s_k \DefinedEqual -\sum_{i=0}^{k-1} a_i g_i$ in
order to obtain dual multipliers $\mu \DefinedEqual (\mu_0, \dots, \mu_{k-1})$.
From \cref{lm-ext-cert}, it follows that
\[
  \max_{x \in \Ball{x_0}{R}}
    \sum_{i=0}^{k-1} (a_i + \mu_i) \DualPairing{g_i}{x_i - x}
  \leq \max_{x \in Q_k}
    \sum_{i=0}^{k-1} a_i \DualPairing{g_i}{x_i - x}
  = \Gamma_k \Delta_k,
\]
(note that $Q_0 = \Omega_0 \cap L_0^- = \Ball{x_0}{R}$). Thus, defining
$\lambda \DefinedEqual a + \mu$, we obtain
$\Gamma_k(\lambda) \equiv \sum_{i=0}^{k-1} \lambda_i \DualNorm{g_i}
\geq \sum_{i=0}^{k-1} a_i \DualNorm{g_i} \equiv \Gamma_k > 0$ and
\[
  \delta_k(\lambda)
  \equiv \max_{x \in \Ball{x_0}{R}} \frac{1}{\Gamma_k(\lambda)}
    \sum_{i=0}^{k-1} \lambda_i \DualPairing{g_i}{x_i - x}
  \leq \frac{\Gamma_k}{\Gamma_k(\lambda)} \Delta_k
  \leq \Delta_k,
\]
Thus, $\lambda$ is a semicertificate whose gap on $\Ball{x_0}{R}$
is bounded by the sliding gap $\Delta_k$.

If $k \geq 0$ is a terminal iteration, then, by the termination criterion and
the definition of $U_k$ (see \cref{alg-gen-sch}), we have $\max_{x \in \Omega_k
\cap L_k^-} \frac{1}{\DualNorm{g_k}} \DualPairing{g_k}{x_k - x} \leq \delta$. In
this case, we apply \cref{alg-augm} to $s_k \DefinedEqual -g_k$ to obtain dual
multipliers $\mu_0, \dots, \mu_{k-1}$. By the same reasoning as above but with
the vector $(0, \dots, 0, 1)$ instead of $(a_0, \dots, a_{k-1})$, we can obtain
that $\delta_{k+1}(\lambda) \leq \delta$, where $\lambda \DefinedEqual (\mu_0,
\dots, \mu_{k-1}, 1)$.

\subsection{Standard Ellipsoid Method}

In the standard Ellipsoid Method, there is no preliminary semicertificate.
Therefore, we cannot apply the above procedure. However, in this method, it is
still possible to generate an accuracy semicertificate although the
corresponding procedure is slightly more involved. Let us now briefly describe
this procedure and discuss how it differs from the previous approach. For
details, we refer the reader to \cite{Nemirovski2010AccCert}.

Let $k \geq 1$ be an iteration of the method. There are two main steps. The
first step is to find a direction $s_k$, in which the ``width'' of the ellipsoid
$\Omega_k$ (see \cref{loc-st-ell}) is minimal:
\[
  s_k \DefinedEqual \adjustlimits \argmin_{\DualNorm{s} = 1}
    \max_{x, y \in \Omega_k} \DualPairing{s}{x - y}
  = \argmin_{\DualNorm{s} = 1} \bigl[ \,
    \max_{x \in \Omega_k} \DualPairing{s}{x}
    - \min_{x \in \Omega_k} \DualPairing{s}{x} \bigr].
\]
It is not difficult to see that $s_k$ is given by the unit
eigenvector\footnote{%
  Here eigenvectors and eigenvalues are defined with respect to the operator $B$
  inducing the norm $\Norm{\cdot}$.%
} %
of the operator $G_k$, corresponding to the largest eigenvalue. For the
corresponding minimal ``width'' of the ellipsoid, we have the following bound
via the average radius:
\begin{equation}\label{std-ell-min-width-ubd}
  \max_{x, y \in \Omega_k} \DualPairing{s_k}{x - y} \leq \rho_k,
\end{equation}
where $\rho_k \DefinedEqual 2 \AverageRadius \Omega_k$. Recall that
$\AverageRadius \Omega_k \leq e^{-k / (2 n^2)} R$ in view of
\cref{std-ell-avrad-final}.

At the second step, we apply \cref{alg-augm} two times with the sets $Q_i
\DefinedEqual \Omega_i$: first, to the vector $s_k$ to obtain dual multipliers
$\mu \DefinedEqual (\mu_0, \dots, \mu_{k-1})$ and then to the vector $-s_k$ to
obtain dual multipliers $\mu' \DefinedEqual (\mu_0', \dots, \mu_{k-1}')$. By
\cref{lm-ext-cert,std-ell-min-width-ubd}, we have
\[
  \begin{aligned}
    \max_{x \in \Ball{x_0}{R}} \Bigl[
    \DualPairing{s_k}{x - x_k}
      + \sum_{i=0}^{k-1} \mu_i \DualPairing{g_i}{x_i - x} \Bigr]
    &\leq \max_{x \in \Omega_k} \DualPairing{s_k}{x - x_k}
    \leq \rho_k, \\
    \max_{x \in \Ball{x_0}{R}} \Bigl[
    \DualPairing{s_k}{x_k - x}
      + \sum_{i=0}^{k-1} \mu_i' \DualPairing{g_i}{x_i - x} \Bigr]
    &\leq \max_{x \in \Omega_k} \DualPairing{s_k}{x_k - x}
    \leq \rho_k
  \end{aligned}
\]
(note that $Q_0 = \Omega_0 = \Ball{x_0}{R}$). Consequently, for $\lambda
\DefinedEqual \mu + \mu'$, we obtain
\[
  \max_{x \in \Ball{x_0}{R}} \sum_{i=0}^{k-1}
    \lambda_i \DualPairing{g_i}{x_i - x} \leq 2 \rho_k.
\]
Finally, one can show that
\[
  \Gamma_k(\lambda)
  \equiv \sum_{i=0}^{k-1} \lambda_i \DualNorm{g_i}
  \geq \frac{r - \rho_k}{D},
\]
where $D$ is the diameter of $Q$ and $r$ is the maximal of the radii of
Euclidean balls contained in $Q$. Thus, whenever $\rho_k < r$, $\lambda$ is a
semicertificate with the following gap on $\Ball{x_0}{R}$:
\[
  \delta_k(\lambda)
  \equiv \max_{x \in \Ball{x_0}{R}}
    \frac{1}{\Gamma_k(\lambda)} \sum_{i=0}^{k-1}
      \lambda_i \DualPairing{g_i}{x_i - x}
  \leq \frac{2 \rho_k D}{r - \rho_k}.
\]

Compared to the standard Ellipsoid Method, we see that, in the Subgradient
Ellipsoid methods, the presence of the preliminary semicertificate removes the
necessity in finding the minimal-``width'' direction and requires only one run
of the Augmentation Algorithm.

%%%%%%%%%%%%%%%%%%%%%%%%%%%%%%%%%%%%%%%%%%%%%%%%%%%%%%%%%%%%%%%%%%%%%%%%%%%%%%%%
%%%%%%%%%%%%%%%%%%%%%%%%%%%%%%%%%%%%%%%%%%%%%%%%%%%%%%%%%%%%%%%%%%%%%%%%%%%%%%%%
\section{Implementation Details}\label{sec-impl-dets}
%%%%%%%%%%%%%%%%%%%%%%%%%%%%%%%%%%%%%%%%%%%%%%%%%%%%%%%%%%%%%%%%%%%%%%%%%%%%%%%%
%%%%%%%%%%%%%%%%%%%%%%%%%%%%%%%%%%%%%%%%%%%%%%%%%%%%%%%%%%%%%%%%%%%%%%%%%%%%%%%%

\subsection{Explicit Representations}
\label{sec-expl-repr}

In the implementation of \cref{alg-gen-sch}, instead of the operators $G_k$, it
is better to work with their inverses $H_k \DefinedEqual G_k^{-1}$. Applying the
Sherman-Morrison formula to \cref{G-next}, we obtain the following update rule
for $H_k$:
\begin{equation}\label{upd-H}
  H_{k+1} = H_k - \frac{b_k H_k g_k g_k\Adjoint H_k}{
    1 + b_k \DualPairing{g_k}{H_k g_k}},
  \quad
  k \geq 0.
\end{equation}

Let us now obtain an explicit formula for the next test point $x_{k+1}$. This
has already been partly done in the proof of \cref{lm-alg-main-ineq}. Indeed,
recall that $x_{k+1}$ is the minimizer of the function $\psi_{k+1}(x)$. From
\cref{psi-next}, we see that $x_{k+1} = x_k - (a_k + \frac{1}{2} b_k U_k)
H_{k+1} g_k$. Combining it with \cref{upd-H}, we obtain
\begin{equation}\label{upd-x}
  x_{k+1} = x_k - \frac{a_k + \frac{1}{2} b_k U_k}{
    1 + b_k \DualPairing{g_k}{H_k g_k}} H_k g_k,
  \quad k \geq 0.
\end{equation}

Finally, one can obtain the following explicit representations for $L_k^-$ and
$\Omega_k$:
\begin{equation}\label{Lm-Omega-expl}
  L_k^- = \SetBuilder{x \in \VectorSpace{E}}{
    \DualPairing{c_k}{x} \leq \sigma_k},
  \qquad
  \Omega_k = \SetBuilder{x \in \VectorSpace{E}}{
    \RelativeNorm{x - z_k}{H_k^{-1}}^2 \leq D_k},
\end{equation}
where
\begin{equation}\label{upd-Lm-Omega-params}
  \begin{aligned}
    c_0 \DefinedEqual 0, \quad \sigma_0 \DefinedEqual 0,
    \qquad
    c_{k+1} \DefinedEqual c_k + a_k g_k,
    \quad
    \sigma_{k+1} \DefinedEqual \sigma_k + a_k \DualPairing{g_k}{x_k},
    \quad k \geq 0, \\
    z_k \DefinedEqual x_k - H_k c_k,
    \quad
    D_k \DefinedEqual R_k^2
      + 2 (\sigma_k - \DualPairing{c_k}{x_k}) + \DualPairing{c_k}{H_k c_k},
    \quad
    k \geq 0.
  \end{aligned}
\end{equation}
Indeed, recalling the definition of functions $\ell_k$, we see that $\ell_k(x) =
\DualPairing{c_k}{x} - \sigma_k$ for all $x \in \VectorSpace{E}$. Therefore,
$L_k^- \equiv \SetBuilder{x}{\ell_k(x) \leq 0} =
\SetBuilder{x}{\DualPairing{c_k}{x} \leq \sigma_k}$. Further, by
\cref{lm-alg-main-ineq}, $\Omega_k = \SetBuilder{x}{\DualPairing{c_k}{x} +
\frac{1}{2} \RelativeNorm{x - x_k}{G_k}^2 \leq \frac{1}{2} R_k^2 + \sigma_k}$.
Note that $\DualPairing{c_k}{x} + \frac{1}{2} \RelativeNorm{x - x_k}{G_k}^2 =
\frac{1}{2} \RelativeNorm{x - z_k}{G_k}^2 + \DualPairing{c_k}{x_k} - \frac{1}{2}
\RelativeDualNorm{c_k}{G_k}^2$ for any $x \in \VectorSpace{E}$. Hence, $\Omega_k
= \SetBuilder{x}{\frac{1}{2} \RelativeNorm{x - z_k}{G_k}^2 \leq \frac{1}{2}
D_k}$.

\begin{remark}\label{rm-st-ell-met}%
  Now we can justify the claim made in \cref{sec-stand-ell-met} that
  \cref{alg-gen-sch} with parameters~\eqref{def-ab}, \eqref{alpha-theta-st-ell},
  \eqref{gamma-st-ell} is the standard Ellipsoid Method. Indeed, from
  \cref{def-ab,loc-st-ell}, we see that $b_k =
  \frac{\gamma}{\DualPairing{g_k}{H_k g_k}}$ and $U_k = R_k
  \DualPairing{g_k}{H_k g_k}^{1/2}$. Also, in view of \cref{gamma-st-ell},
  $\frac{\gamma}{1 + \gamma} = \frac{2}{n + 1}$. Hence, by \cref{upd-x,upd-H},
  \begin{equation}\label{upd-st-ell}
    \begin{aligned}
      x_{k+1} &= x_k - \frac{R_k}{n + 1}
        \frac{H_k g_k}{\DualPairing{g_k}{H_k g_k}^{1/2}}, \\
      H_{k+1} &= H_k - \frac{2}{n + 1}
        \frac{H_k g_k g_k\Adjoint H_k}{\DualPairing{g_k}{H_k g_k}},
      \qquad k \geq 0.
    \end{aligned}
  \end{equation}
  Further, according to \cref{R-sqr-st-ell,gamma-st-ell}, for any $k \geq 0$, we
  have $R_k^2 = q^k R^2$, where $q = 1 + \frac{1}{(n - 1) (n + 1)} =
  \frac{n^2}{n^2 - 1}$. Thus, method~\eqref{upd-st-ell} indeed
  coincides\footnote{%
    Note that, in \cref{st-ell-met}, we identify the spaces $\VectorSpace{E}$,
    $\VectorSpace{E}\Dual$ with $\RealField^n$ in such a way that
    $\DualPairing{\cdot}{\cdot}$ coincides with the standard dot-product and
    $\Norm{\cdot}$ coincides with the standard Euclidean norm. Therefore, $B$
    becomes the identity matrix and $g_k\Adjoint$ becomes $g_k\Transpose$.%
  } %
  with the standard Ellipsoid Method \eqref{st-ell-met} under the change of
  variables $W_k \DefinedEqual R_k^2 H_k$.
\end{remark}

\subsection{Computing Support Function}\label{sec-comp-U}

To calculate $U_k$ in \cref{alg-gen-sch}, we need to compute the following
quantity (see \cref{Lm-Omega-expl}):
\[
  U_k = \max_x \SetBuilder{\DualPairing{g_k}{x_k - x}}{
    \RelativeNorm{x - z_k}{H_k^{-1}}^2 \leq D_k,
    \DualPairing{c_k}{x} \leq \sigma_k}.
\]
Let us discuss how to do this.

First, let us introduce the following support function to simplify our notation:
\[
  \xi(H, s, a, \beta) \DefinedEqual \max_x \SetBuilder{\DualPairing{s}{x}}{
    \RelativeNorm{x}{H^{-1}}^2 \leq 1,
    \DualPairing{a}{x} \leq \beta},
\]
where $\Map{H}{\VectorSpace{E}\Dual}{\VectorSpace{E}}$ is a self-adjoint
positive definite linear operator, $s, a \in \VectorSpace{E}\Dual$ and $\beta
\in \RealField$. In this notation, assuming that $D_k > 0$, we have
\[
  U_k = \DualPairing{g_k}{x_k - z_k}
    + \xi(D_k H_k, -g_k, c_k, \sigma_k - \DualPairing{c_k}{z_k}).
\]

Let us show how to compute $\xi(H, s, a, \beta)$. Dualizing the linear
constraint, we obtain
\begin{equation}\label{aux-prob}
  \xi(H, s, a, \beta) = \min_{\tau \geq 0} \bigl[
    \RelativeDualNorm{s - \tau a}{H^{-1}} + \tau \beta \bigr],
\end{equation}
provided that there exists some $x \in \VectorSpace{E}$ such that
$\RelativeNorm{x}{H^{-1}} < 1$, $\DualPairing{a}{x} \leq \beta$ (Slater
condition). One can show that \eqref{aux-prob} has the following solution (see
\cref{lm-aux-op}):
\begin{equation}\label{opt-tau}
  \tau(H, s, a, \beta) \DefinedEqual
  \begin{cases}
    0, & \text{if
      $\DualPairing{a}{H s} \leq \beta \RelativeDualNorm{s}{H^{-1}}$}, \\
    u(H, s, a, \beta), & \text{otherwise},
  \end{cases}
\end{equation}
where $u(H, s, a, \beta)$ is the unconstrained minimizer of the objective
function in \cref{aux-prob}.

Let us present an explicit formula for $u(H, s, a, \beta)$. For future use, it
will be convenient to write down this formula in a slightly more general form
for the following multidimensional\footnote{%
  Hereinafter, we identify $(\RealField^m)\Dual$ with $\RealField^m$ in such a
  way that $\DualPairing{\cdot}{\cdot}$ is the standard dot product.%
} %
variant of problem~\eqref{aux-prob}:
\begin{equation}\label{aux-uncon}
  \min_{u \in \RealField^m} \bigl[
    \RelativeDualNorm{s - A u}{H^{-1}} + \DualPairing{u}{b} \bigr],
\end{equation}
where $s \in \VectorSpace{E}\Dual$,
$\Map{H}{\VectorSpace{E}\Dual}{\VectorSpace{E}}$ is a self-adjoint positive
definite linear operator, $\Map{A}{\RealField^m}{\VectorSpace{E}\Dual}$ is a
linear operator with trivial kernel and $b \in \RealField^m$,
$\DualPairing{b}{(A\Adjoint H A)^{-1} b} < 1$. It is not difficult to show that
problem~\eqref{aux-uncon} has the following unique solution (see
\cref{lm-uncon-min}):
\begin{equation}\label{uncon-min}
  u(H, s, A, b) \DefinedEqual (A\Adjoint H A)^{-1} (A\Adjoint s - r b),
  \quad
  r \DefinedEqual \sqrt{\frac{
    \DualPairing{s}{H s}
    - \DualPairing{s}{A (A\Adjoint H A)^{-1} A\Adjoint s}
  }{1 - \DualPairing{b}{(A\Adjoint H A)^{-1} b}}}.
\end{equation}

Note that, in order for the above approach to work, we need to guarantee that
the sets $\Omega_k$ and $L_k^-$ satisfy a certain regularity condition, namely,
$\Interior \Omega_k \cap L_k^- \neq \emptyset$. This condition can be easily
fulfilled by adding into \cref{alg-gen-sch} the termination
criterion~\eqref{term-U}.

\begin{lemma}\label{lm-slater-cond}%
  Consider \cref{alg-gen-sch} with termination criterion~\eqref{term-U}. Then,
  at each iteration $k \geq 0$, at the beginning of
  \cref{alg:NewEllipsoidMethod:ComputeU}, we have $\Interior \Omega_k \cap L_k^-
  \neq \emptyset$. Moreover, if $k$ is a nonterminal iteration, we also have
  $\DualPairing{g_k}{x - x_k} \leq 0$ for some $x \in \Interior \Omega_k \cap
  L_k^-$.
\end{lemma}

\begin{proof}
  Note that $\Interior \Omega_0 \cap L_0^- = \Interior \Ball{x_0}{R} \neq
  \emptyset$. Now suppose $\Interior \Omega_k \cap L_k^- \neq \emptyset$ for
  some nonterminal iteration $k \geq 0$. Denote $P_k^- \DefinedEqual
  \SetBuilder{x \in \VectorSpace{E}}{\DualPairing{g_k}{x - x_k} \leq 0}$. Since
  iteration $k$ is nonterminal, $U_k > 0$ and hence $\Omega_k \cap L_k^- \cap
  \Interior P_k^- \neq \emptyset$. Combining it with the fact that $\Interior
  \Omega_k \cap L_k^- \neq \emptyset$, we obtain $\Interior \Omega_k \cap L_k^-
  \cap \Interior P_k^- \neq \emptyset$ and, in particular, $\Interior \Omega_k
  \cap L_k^- \cap P_k^- \neq \emptyset$. At the same time, slightly modifying
  the proof of \cref{lm-set-inc} (using that $\Interior \Omega_i = \SetBuilder{x
  \in \VectorSpace{E}}{\omega_i(x) < \frac{1}{2} R^2}$ for any $i \geq 0$ since
  $\omega_i$ is a strictly convex quadratic function), it is not difficult to
  show that $\Interior \Omega_k \cap L_k^- \cap P_k^- \subseteq \Interior
  \Omega_{k+1} \cap L_{k+1}^-$. Thus, $\Interior \Omega_{k+1} \cap L_{k+1}^-
  \neq \emptyset$, and we can continue by induction.
\end{proof}

\subsection{Computing Dual Multipliers}\label{sec-comp-dual-mult}

Recall from \cref{sec-gen-acc-cert} that the procedure for generating an
accuracy semicertificate for \cref{alg-gen-sch} requires one to repeatedly carry
out the following operation: given $s \in \VectorSpace{E}\Dual$ and some
iteration number $i \geq 0$, compute a dual multiplier $\mu \geq 0$ such that
\[
  \max_{x \in \Omega_i \cap L_i^-}
    \SetBuilder{\DualPairing{s}{x}}{\DualPairing{g_i}{x - x_i} \leq 0}
  = \max_{x \in \Omega_i \cap L_i^-} \bigl[
      \DualPairing{s}{x} + \mu \DualPairing{g_i}{x_i - x} \bigr].
\]
This can be done as follows.

First, using \cref{Lm-Omega-expl}, let us rewrite the above primal problem more
explicitly:
\[
  \max_x \SetBuilder{\DualPairing{s}{x}}{
    \RelativeNorm{x - z_i}{H_i^{-1}}^2 \leq D_i,
    \DualPairing{c_i}{x} \leq \sigma_i,
    \DualPairing{g_i}{x - x_i} \leq 0}.
\]
Our goal is to dualize the second linear constraint and find the corresponding
multiplier. However, for the sake of symmetry, it is better to dualize both
linear constraints, find the corresponding multipliers and then keep only the
second one.

Let us simplify our notation by introducing the following problem:
\begin{equation}\label{pri-prob}
  \max_x \SetBuilder{\DualPairing{s}{x}}{
    \RelativeNorm{x}{H^{-1}} \leq 1,
    \DualPairing{a_1}{x} \leq b_1,
    \DualPairing{a_2}{x} \leq b_2},
\end{equation}
where $\Map{H}{\VectorSpace{E}\Dual}{\VectorSpace{E}}$ is a self-adjoint
positive definite linear operator, $s, a_1, a_2 \in \VectorSpace{E}\Dual$ and
$b_1, b_2 \in \RealField$. Clearly, our original problem can be transformed into
this form by setting $H \DefinedEqual D_i H_i$, $a_1 \DefinedEqual c_i$, $a_2
\DefinedEqual g_i$, $b_1 \DefinedEqual \sigma_i - \DualPairing{c_i}{z_i}$, $b_2
\DefinedEqual \DualPairing{g_i}{x_i - z_i}$. Note that this transformation does
not change the dual multipliers.

Dualizing the linear constraints in \eqref{pri-prob}, we obtain the following
dual problem:
\begin{equation}\label{dual-prob}
  \min_{\mu \in \NonnegativeRay^2} \bigl[
    \RelativeDualNorm{s - \mu_1 a_1 - \mu_2 a_2}{H^{-1}}
    + \mu_1 b_1 + \mu_2 b_2 \bigr],
\end{equation}
which is solvable provided the following Slater condition holds:
\begin{equation}\label{Slater}
  \exists x \in \VectorSpace{E} \colon
  \RelativeNorm{x}{H^{-1}} < 1,
  \DualPairing{a_1}{x} \leq b_1,
  \DualPairing{a_2}{x} \leq b_2.
\end{equation}
Note that \eqref{Slater} can be ensured by adding termination
criterion~\eqref{term-U} into \cref{alg-gen-sch} (see \cref{lm-slater-cond}).

A solution of \eqref{dual-prob} can be found using \cref{alg-dual-mult}. In this
routine, $\tau(\cdot)$, $\xi(\cdot)$ and $u(\cdot)$ are the auxiliary
operations, defined in \cref{sec-comp-U}, and $A \DefinedEqual (a_1, a_2)$ is
the linear operator $A u \DefinedEqual u_1 a_1 + u_2 a_2$ acting from
$\RealField^2$ to $\VectorSpace{E}\Dual$. The correctness of
\cref{alg-dual-mult} is proved in \cref{th-alg-dual-mult-cor}.

\begin{SimpleAlgorithm}[%
  title={Computing Dual Multipliers},%
  label=alg-dual-mult,%
  width=0.75\linewidth%
]%
  \begin{AlgorithmGroup}
    \begin{AlgorithmSteps}
      \AlgorithmStep\label{step:comp-tau}%
        Compute $\tau_1 \DefinedEqual \tau(H, s, a_1, b_1)$
          and $\tau_2 \DefinedEqual \tau(H, s, a_2, b_2)$.\\
        Compute $\xi_1 \DefinedEqual \xi(H, a_2, a_1, b_1)$
          and $\xi_2 \DefinedEqual \xi(H, a_1, a_2, b_2)$.

      \AlgorithmStep\label{step:check-set-inc}%
        If $\xi_1 \leq b_2$, return $(\tau_1, 0)$.
        Else if $\xi_2 \leq b_1$, return $(0, \tau_2)$.

      \AlgorithmStep\label{step:check-points}%
        Else if $\DualPairing{a_2}{H (s - \tau_1 a_1)}
          \leq b_2 \RelativeDualNorm{s - \tau_1 a_1}{H^{-1}}$,
        return $(\tau_1, 0)$. \\
        Else if $\DualPairing{a_1}{H (s - \tau_2 a_2)}
          \leq b_1 \RelativeDualNorm{s - \tau_2 a_2}{H^{-1}}$,
        return $(0, \tau_2)$.
        
      \AlgorithmStep\label{step:final}%
        Else return $u \DefinedEqual u(H, s, A, b)$, where $A \DefinedEqual
        (a_1, a_2)$, $b \DefinedEqual (b_1, b_2)\Transpose$.
    \end{AlgorithmSteps}
  \end{AlgorithmGroup}
\end{SimpleAlgorithm}

\subsection{Time and Memory Requirements}

Let us discuss the time and memory requirements of \cref{alg-gen-sch}, taking
into account the the previously mentioned implementation details.

The main objects in \cref{alg-gen-sch}, which need to be stored and updated
between iterations, are the test points $x_k$, matrices $H_k$, scalars $R_k$,
vectors $c_k$ and scalars $\sigma_k$, see
\cref{upd-x,upd-H,R-rec,upd-Lm-Omega-params} for the corresponding updating
formulas. To store all these objects, we need $\BigO(n^2)$ memory.

Consider now what happens at each iteration $k$. First, we compute $U_k$. For
this, we calculate $z_k$ and $D_k$ according to \cref{upd-Lm-Omega-params} and
then perform the calculations described in \cref{sec-comp-U}. The most difficult
operation there is computing the matrix-vector product, which takes $\BigO(n^2)$
time. After that, we calculate the coefficients $a_k$ and $b_k$ according to
\cref{def-ab}, where $\alpha_k$, $\theta$ and $\gamma$ are certain scalars,
easily computable for all main instances of \cref{alg-gen-sch} (see \cref{%
  sec-subgrad-met,%
  sec-stand-ell-met,%
  sec-ell-met-w-prel-cert,%
  sec-subgrad-ell-met%
}). The most expensive step there is computing the norm
$\RelativeDualNorm{g_k}{G_k}$, which can be done in $\BigO(n^2)$ operations by
evaluating the product $H_k g_k$. Finally, we update our main objects, which
takes $\BigO(n^2)$ time.

Thus, each iteration of \cref{alg-gen-sch} has $\BigO(n^2)$ time and memory
complexities, exactly as in the standard Ellipsoid Method.

Now let us analyze the complexity of the auxiliary procedure from
\cref{sec-gen-acc-cert} for converting a preliminary semicertificate into a
semicertificate. The main operation in this procedure is running
\cref{alg-augm}, which iterates ``backwards'', computing some dual multiplier
$\mu_i$ at each iteration $i = k-1, \dots, 0$. Using the approach from
\cref{sec-comp-dual-mult}, we can compute $\mu_i$ in $\BigO(n^2)$ time, provided
that the objects $x_i$, $g_i$, $H_i$, $z_i$, $D_i$, $c_i$, $\sigma_i$ are stored
in memory. Note, however, that, in contrast to the ``forward'' pass, when
iterating ``backwards'', there is no way to efficiently recompute all these
objects without storing in memory a certain ``history'' of the main process from
iteration $0$ up to $k$. The simplest choice is to keep in this ``history'' all
the objects mentioned above, which requires $\BigO(k n^2)$ memory. A slightly
more efficient idea is to keep the matrix-vector products $H_i g_i$ instead of
$H_i$ and then use \cref{upd-H} to recompute $H_i$ from $H_{i+1}$ in
$\BigO(n^2)$ operations. This allows us to reduce the size of the ``history''
down to $\BigO(k n)$ while still keeping the $\BigO(k n^2)$ total time
complexity of the auxiliary procedure. Note that these estimates are exactly the
same as those for the best currently known technique for generating accuracy
certificates in the standard Ellipsoid Method
\cite{Nemirovski2010AccCert}. In particular, if we generate a
semicertificate only once at the very end, then the time complexity of our
procedure is comparable to that of running the standard Ellipsoid Method without
computing any certificates. Alternatively, as suggested
in~\cite{Nemirovski2010AccCert}, one can generate semicertificates,
say, every $2, 4, 8, 16, \dots$ iterations. Then, the total ``overhead'' of the
auxiliary procedure for generating semicertificates will be comparable to the
time complexity of the method itself.

%%%%%%%%%%%%%%%%%%%%%%%%%%%%%%%%%%%%%%%%%%%%%%%%%%%%%%%%%%%%%%%%%%%%%%%%%%%%%%%%
%%%%%%%%%%%%%%%%%%%%%%%%%%%%%%%%%%%%%%%%%%%%%%%%%%%%%%%%%%%%%%%%%%%%%%%%%%%%%%%%
\section{Conclusion}\label{sec-concl}
%%%%%%%%%%%%%%%%%%%%%%%%%%%%%%%%%%%%%%%%%%%%%%%%%%%%%%%%%%%%%%%%%%%%%%%%%%%%%%%%
%%%%%%%%%%%%%%%%%%%%%%%%%%%%%%%%%%%%%%%%%%%%%%%%%%%%%%%%%%%%%%%%%%%%%%%%%%%%%%%%

In this paper, we have addressed one of the issues of the standard Ellipsoid
Method, namely, its poor convergence for problems of large dimension $n$. For
this, we have proposed a new algorithm which can be seen as the combination of
the Subgradient and Ellipsoid methods.

Note that there are still some open questions. First, the convergence estimate
of our method with time-varying coefficients contains an extra factor
proportional to the logarithm of the iteration counter. We have seen that this
logarithmic factor has its roots yet in the Subgradient Method. However, as
discussed in \cref{rm-subgrad-met-wo-log-fact}, for the Subgradient Method, this
issue can be easily resolved by allowing projections onto the feasible set and
working with ``truncated'' gaps. An even better alternative, which does not
require any of this machinery, is to use Dual Averaging
\cite{Nesterov2009PrimalDual} instead of the Subgradient Method. It is an
interesting question whether one can combine the Dual Averaging with the
Ellipsoid Method similarly to how we have combined the Subgradient and Ellipsoid
methods.

Second, the convergence rate estimate, which we have obtained for our method, is
not continuous in the dimension $n$. Indeed, for small values of the iteration
counter~$k$, this estimate behaves as that of the Subgradient Method and then,
at some moment (around $n^2$), it switches to the estimate of the Ellipsoid
Method. As discussed at the end of \cref{sec-subgrad-ell-met}, there exists some
``small'' gap between these two estimates around the switching moment.
Nevertheless, the method itself is continuous in $n$ and does not contain any
explicit switching rules. Therefore, there should be some continuous convergence
rate estimate for our method and it is an open question to find it.

Finally, apart from the Ellipsoid Method, there exist other
``dimension-dependent`` methods (e.g., the Center-of-Gravity Method\footnote{%
Although this method is not practical, it is still interested from an academic
point of view.%
}%
~\cite{Levin1965,Newman1965}, the Inscribed Ellipsoid
Method~\cite{TarasovKhachiyan1988}, the Circumscribed Simplex
Method~\cite{Bulatov1982}, etc.). Similarly, the Subgradient Method is not the
only ``dimension-independent'' method and there exist numerous alternatives
which are better suited for certain problem classes (e.g., the Fast Gradient
Method~\cite{Nesterov1983} for Smooth Convex Optimization or methods for
Stochastic Programming~\cite{%
  Nemirovski2009Robust,%
  Duchi2011Adagrad,%
  Lan2012Optimal,%
  Dvurechensky2016Stochastic%
}). Of course, it is interesting to consider different combinations of the
aforementioned ``dimension-dependent'' and ``dimension-independent'' methods. In
this regard, it is also worth mentioning the
works~\cite{Bubeck2015Geometric,Bubeck2016Politician}, where the authors propose
new variants of gradient-type methods for smooth strongly convex minimization
problems inspired by the geometric construction of the Ellipsoid Method.

%%%%%%%%%%%%%%%%%%%%%%%%%%%%%%%%%%%%%%%%%%%%%%%%%%%%%%%%%%%%%%%%%%%%%%%%%%%%%%%%
%%%%%%%%%%%%%%%%%%%%%%%%%%%%%%%%%%%%%%%%%%%%%%%%%%%%%%%%%%%%%%%%%%%%%%%%%%%%%%%%

\appendix

%%%%%%%%%%%%%%%%%%%%%%%%%%%%%%%%%%%%%%%%%%%%%%%%%%%%%%%%%%%%%%%%%%%%%%%%%%%%%%%%
%%%%%%%%%%%%%%%%%%%%%%%%%%%%%%%%%%%%%%%%%%%%%%%%%%%%%%%%%%%%%%%%%%%%%%%%%%%%%%%%
\section{%
  Proof of
  \texorpdfstring{\cref{lm-min-zeta}}{\crefName{lemma}~\ref{lm-min-zeta}}%
}
\label{sec-proof-lm-min-zeta}
%%%%%%%%%%%%%%%%%%%%%%%%%%%%%%%%%%%%%%%%%%%%%%%%%%%%%%%%%%%%%%%%%%%%%%%%%%%%%%%%
%%%%%%%%%%%%%%%%%%%%%%%%%%%%%%%%%%%%%%%%%%%%%%%%%%%%%%%%%%%%%%%%%%%%%%%%%%%%%%%%

\begin{proof}
  Everywhere in the proof, we assume that the parameter $c$ is fixed and drop
  all the indices related to it.
  
  Let us show that $\zeta_p$ is a convex function. Indeed, the function
  $\Map{\omega}{\RealField \times \PositiveRay}{\RealField}$, defined by
  $\omega(x, t) \DefinedEqual \frac{x^2}{t}$, is convex. Hence, the function
  $q$, defined in \cref{def-zeta-q}, is also convex. Further, since $\omega$ is
  increasing in its first argument on $\NonnegativeRay$, the function
  $\Map{\omega_p}{\NonnegativeRay \times \PositiveRay}{\RealField}$, defined by
  $\omega_p(x, t) \DefinedEqual \frac{x^p}{t}$, is also convex as the
  composition of $\omega$ with the mapping $(x, t) \mapsto (x^{p/2}, t)$, whose
  first component is convex (since $p \geq 2$) and the second one is affine.
  Note that $\omega_p$ is increasing in its first argument. Hence, $\zeta_p$ is
  indeed a convex function as the composition of $\omega_p$ with the mapping
  $\gamma \mapsto \bigl(q(\gamma), 1 + \gamma \bigr)$, whose first part is
  convex and the second one is affine.

  Differentiating, for any $\gamma > 0$, we obtain
  \[
    \zeta_p'(\gamma) = \frac{
        p [q(\gamma)]^{p-1} q'(\gamma) (1 + \gamma)
        - [q(\gamma)]^p
      }{(1 + \gamma)^2}
    = \frac{
        [q(\gamma)]^{p-1} \bigl(
          p q'(\gamma) (1 + \gamma) - q(\gamma)
        \bigr)}{(1 + \gamma)^2}.
  \]
  Therefore, the minimizers of $\zeta_p$ are exactly solutions to the following
  equation:
  \begin{equation}\label{opt-cond-gamma}
    p q'(\gamma) (1 + \gamma) = q(\gamma).
  \end{equation}

  Note that $q'(\gamma) = \frac{c [2 \gamma (1 + \gamma) - \gamma^2]}{2 (1 +
  \gamma)^2} = \frac{c \gamma (2 + \gamma)}{2 (1 + \gamma)^2}$ (see
  \cref{def-zeta-q}). Hence, \cref{opt-cond-gamma} can be written as $c p \gamma
  (2 + \gamma) = 2 (1 + \gamma) + c \gamma^2$ or, equivalently, $c (p - 1)
  \gamma^2 + 2 (c p - 1) \gamma = 2$. Clearly, $\gamma = 0$ is not a solution of
  this equation. Making the change of variables $\gamma = \frac{2}{u}$, $u \neq
  0$, we come the quadratic equation $u^2 - 2 (c p - 1) u = 2 c (p - 1)$ or,
  equivalently, to $[ u - (c p - 1) ]^2 = 2 c (p - 1) + (c p - 1)^2 = c^2 p^2 -
  (2 c - 1)$. This equation has two solutions: $u_1 \DefinedEqual c p - 1 +
  \sqrt{c^2 p^2 - (2 c - 1)}$ and $u_2 \DefinedEqual c p - 1 - \sqrt{c^2 p^2 -
  (2 c - 1)}$. Note that $u_2 \geq c p - 1 - \sqrt{c^2 p^2 + 1} \geq c p - 1 -
  (c p + 1) = -2$. Hence, $\gamma_2 \DefinedEqual \frac{2}{u_2} \leq -1$ cannot
  be a minimizer of $\zeta_p$. Consequently, only $u_1$ is an acceptable
  solution (note that $u_1 > 0$ in view of our assumptions on $c$ and $p$).
  Thus, \cref{gamma-opt} is proved.

  Let us show that $\gamma(p)$ belongs to the interval specified in
  \cref{gamma-opt}. For this, we need to prove that $1 \leq c p \gamma(p) \leq
  2$. Note that the function $h_a(t) \DefinedEqual \frac{t}{\sqrt{t^2 - a} + t -
  1}$, where $a \geq 0$, is decreasing in $t$. Indeed, $\frac{1}{h_a(t)} =
  \sqrt{1 - \frac{a}{t^2}} - \frac{1}{t} + 1$ is an increasing function in $t$.
  Hence, $c p \gamma(p) = 2 h_{2 c - 1}(c p) \geq 2 \lim_{t \to \infty} h_{2 c -
  1}(t) = 1$. On the other hand, using that $p \geq 2$ and denoting $\alpha
  \DefinedEqual 2 c \geq 1$, we get $c p \gamma(p) = 2 h_{\alpha - 1}(c p) \leq
  2 g(\alpha)$, where $g(\alpha) \DefinedEqual h_{\alpha - 1}(\alpha) =
  \frac{\alpha}{\sqrt{\alpha^2 - \alpha + 1} + \alpha - 1}$. Note that $g$ is
  decreasing in $\alpha$. Indeed, denoting $\tau \DefinedEqual \frac{1}{\alpha}
  \in \OpenClosedInterval{0}{1}$, we obtain $\frac{1}{g(\alpha)} = \sqrt{1 -
  \tau + \tau^2} - \tau + 1$, which is a decreasing function in $\tau$. Thus, $c
  p \gamma(p) \leq 2 g(1) = 2$.

  It remains to prove that $\zeta_p(\gamma(p)) \leq e^{-1 / (2 c p)}$. Let
  $\Map{\phi}{\ClosedOpenInterval{2}{+\infty}}{\RealField}$ be the function
  \begin{equation}\label{def-phi-p}
    \phi(p)
    \DefinedEqual -\ln \zeta_p\bigl( \gamma(p) \bigr)
    = \ln\bigl( 1 + \gamma(p) \bigr) - p \ln q\bigl( \gamma(p) \bigr).
  \end{equation}
  We need to show that $\phi(p) \geq 1 / (2 c p)$ for all $p \geq 2$ or,
  equivalently, that the function
  $\Map{\chi}{\OpenClosedInterval{0}{\frac{1}{2}}}{\RealField}$, defined by
  $\chi(\tau) \DefinedEqual \phi(\frac{1}{\tau})$, satisfies $\chi(\tau) \geq
  \frac{\tau}{2 c}$ for all $\tau \in \OpenClosedInterval{0}{\frac{1}{2}}$. For
  this, it suffices to show that $\chi$ is convex, $\lim_{\tau \to 0} \chi(\tau)
  = 0$ and $\lim_{\tau \to 0} \chi'(\tau) = \frac{1}{2 c}$. Differentiating, we
  see that $\chi'(\tau) = -\frac{1}{\tau^2} \phi'(\frac{1}{\tau} )$ and
  $\chi''(\tau) = \frac{2}{\tau^3} \phi'(\frac{1}{\tau} ) + \frac{1}{\tau^4}
  \phi''(\frac{1}{\tau} )$ for all $\tau \in
  \OpenClosedInterval{0}{\frac{1}{2}}$. Thus, we need to justify that
  \begin{equation}\label{conv-chi-via-phi}
    2 \phi'(p) + p \phi''(p) \geq 0
  \end{equation}
  for all $p \geq 2$ and that
  \begin{equation}\label{lim-phi}
    \lim_{p \to \infty} \phi(p) = 0,
    \qquad
    \lim_{p \to \infty} [-p^2 \phi'(p)] = \frac{1}{2 c}.
  \end{equation}

  Let $p \geq 2$ be arbitrary. Differentiating and using \cref{opt-cond-gamma},
  we obtain
  \begin{equation}\label{der-phi-p}
    \begin{aligned}
      \phi'(p) &= \frac{\gamma'(p)}{1 + \gamma(p)}
      - \ln q\bigl( \gamma(p) \bigr)
      - \frac{p q'\bigl( \gamma(p) \bigr) \gamma'(p)}{
          q\bigl( \gamma(p) \bigr)}
      = -\ln q\bigl( \gamma(p) \bigr), \\
      \phi''(p) &= -\frac{q'(\gamma(p)) \gamma'(p)}{q(\gamma(p))}
      = -\frac{\gamma'(p)}{p \bigl( 1 + \gamma(p) \bigr)}.
    \end{aligned}
  \end{equation}
  Therefore,
  \[
    2 \phi'(p) + p \phi''(p)
    = -2 \ln q(\gamma(p)) - \frac{\gamma'(p)}{1 + \gamma(p)}
    \geq -\frac{c \gamma^2(p) + \gamma'(p)}{1 + \gamma(p)},
  \]
  where the inequality follows from \cref{def-zeta-q} and the fact that $\ln(1 +
  \tau) \leq \tau$ for any $\tau > -1$. Thus, to show \cref{conv-chi-via-phi},
  we need to prove that $-\gamma'(p) \geq c \gamma^2(p)$ or, equivalently,
  $\frac{d}{d p} \frac{1}{\gamma(p)} \geq c$. But this is immediate. Indeed,
  using \cref{gamma-opt}, we obtain $\frac{d}{d p} \frac{1}{\gamma(p)} =
  \frac{c}{2} ( \frac{c p}{\sqrt{c^2 p^2 - (2 c - 1)}} + 1 ) \geq c$ since the
  function $\tau \mapsto \frac{\tau}{\sqrt{\tau^2 - 1}}$ is decreasing. Thus,
  \cref{conv-chi-via-phi} is proved.

  It remains to show \cref{lim-phi}. From \cref{gamma-opt}, we see that
  $\gamma(p) \to 0$ and $p \gamma(p) \to \frac{1}{c}$ as $p \to \infty$. Hence,
  using \cref{def-zeta-q}, we obtain
  \[
    \lim_{p \to \infty} p^2 \ln q\bigl( \gamma(p) \bigr)
    = \lim_{p \to \infty} \frac{c p^2 \gamma^2(p)}{
      2 \bigl( 1 + \gamma(p) \bigr)}
    = \frac{c}{2} \lim_{p \to \infty} p^2 \gamma^2(p)
    = \frac{1}{2 c}.
  \]
  Consequently, in view of \cref{def-phi-p,der-phi-p}, we have
  \[
    \begin{aligned}
      \lim_{p \to \infty} \phi(p)
      &= \lim_{p \to \infty} \bigl[ \ln\bigl( 1 + \gamma(p) \bigr)
        - p \ln q\bigl( \gamma(p) \bigr) \bigr]
      = 0, \\
      \lim_{p \to \infty} [-p^2 \phi'(p)]
      &= \lim_{p \to \infty} p^2 \ln q\bigl( \gamma(p) \bigr)
      = \frac{1}{2 c}.
    \end{aligned}
  \]
  which is exactly \cref{lim-phi}.
\end{proof}

%%%%%%%%%%%%%%%%%%%%%%%%%%%%%%%%%%%%%%%%%%%%%%%%%%%%%%%%%%%%%%%%%%%%%%%%%%%%%%%%
%%%%%%%%%%%%%%%%%%%%%%%%%%%%%%%%%%%%%%%%%%%%%%%%%%%%%%%%%%%%%%%%%%%%%%%%%%%%%%%%
\section{Support Function and Dual Multipliers: Proofs}
%%%%%%%%%%%%%%%%%%%%%%%%%%%%%%%%%%%%%%%%%%%%%%%%%%%%%%%%%%%%%%%%%%%%%%%%%%%%%%%%
%%%%%%%%%%%%%%%%%%%%%%%%%%%%%%%%%%%%%%%%%%%%%%%%%%%%%%%%%%%%%%%%%%%%%%%%%%%%%%%%

For brevity, everywhere in this section, we write $\Norm{\cdot}$ and
$\DualNorm{\cdot}$ instead of $\RelativeNorm{\cdot}{H^{-1}}$ and
$\RelativeDualNorm{\cdot}{H^{-1}}$, respectively. We also denote $B_0
\DefinedEqual \SetBuilder{x \in \VectorSpace{E}}{\Norm{x} \leq 1}$.

\subsection{Auxiliary Operations}

\begin{lemma}\label{lm-uncon-min}%
  Let $s \in \VectorSpace{E}\Dual$, let
  $\Map{A}{\RealField^m}{\VectorSpace{E}\Dual}$ be a linear operator with
  trivial kernel and let $b \in \RealField^m$, $\DualPairing{b}{(A\Adjoint H
  A)^{-1} b} < 1$. Then, \eqref{aux-uncon} has a unique solution given by
  \cref{uncon-min}.
\end{lemma}

\begin{proof}
  Note that the sublevel sets of the objective function in \cref{aux-uncon} are
  bounded:
  \[
    \DualNorm{s - A u} + \DualPairing{u}{b}
    \geq \DualNorm{A u} - \DualNorm{s} + \DualPairing{u}{b}
    \geq (1 - \DualPairing{b}{(A\Adjoint H A)^{-1} b}^{1/2}) \DualNorm{A u}
      - \DualNorm{s}
  \]
  for all $u \in \RealField^m$. Hence, problem~\eqref{aux-uncon} has a solution.

  Let $u \in \RealField^m$ be a solution of problem~\eqref{aux-uncon}. If $s = A
  u$, then $u = (A\Adjoint H A)^{-1} A\Adjoint s$, which coincides with the
  solution given by \cref{uncon-min} (note that, in this case, $r = 0$).
  
  Now suppose $s \neq A u$. Then, from the first-order optimality condition, we
  obtain that $b = A\Adjoint (s - A u) / \rho$, where $\rho \DefinedEqual
  \DualNorm{s - A u} > 0$. Hence, $u = (A\Adjoint H A)^{-1} (A\Adjoint s - \rho
  b)$ and
  \[
    \begin{aligned}
      \rho^2
      &= \DualNorm{s - A u}^2
      = \DualNorm{s}^2 - 2 \DualPairing{A\Adjoint s}{u}
        + \DualPairing{A\Adjoint H A u}{u} \\
      &= \DualNorm{s}^2 - 2 \DualPairing{A\Adjoint s}{
          (A\Adjoint H A)^{-1} (A\Adjoint s - \rho b)}
        + \DualPairing{A\Adjoint s - \rho b}{
          (A\Adjoint H A)^{-1} (A\Adjoint s - \rho b)} \\
      &= \DualNorm{s}^2 - \DualPairing{s}{A (A\Adjoint H A)^{-1} A\Adjoint s}
        + \rho^2 \DualPairing{b}{(A\Adjoint H A)^{-1} b}. 
    \end{aligned}
  \]
  Thus, $\rho = r$ and $u = u(H, s, A, b)$ given by \cref{uncon-min}.
\end{proof}

\begin{lemma}\label{lm-aux-op}%
  Let $s, a \in \VectorSpace{E}\Dual$, $\beta \in \RealField$ be such that
  $\DualPairing{a}{x} \leq \beta$ for some $x \in \Interior B_0$. Then,
  \eqref{aux-prob} has a solution given by \cref{opt-tau}. Moreover, this
  solution is unique if $\beta < \DualNorm{a}$.
\end{lemma}

\begin{proof}
  Let $\Map{\phi}{\RealField}{\RealField}$ be the function $\phi(\tau)
  \DefinedEqual \DualNorm{s - \tau a} + \tau \beta$. By our assumptions, $\beta
  > -\DualNorm{a}$ if $a \neq 0$ and $\beta \geq 0$ if $a = 0$. If additionally
  $\beta < \DualNorm{a}$, then $\AbsoluteValue{\beta} < \DualNorm{a}$.

  If $s = 0$, then $\phi(\tau) = \tau (\DualNorm{a} + \beta) \geq \phi(0)$ for
  all $\tau \geq 0$, so $0$ is a solution of~\eqref{aux-prob}. Clearly, this
  solution is unique when $\beta < \DualNorm{a}$ because then
  $\AbsoluteValue{\beta} < \DualNorm{a}$.

  From now on, suppose $s \neq 0$. Then, $\phi$ is differentiable at $0$ with
  $\phi'(0) = \beta - \DualPairing{a}{s} / \DualNorm{s}$. If $\DualPairing{a}{s}
  \leq \beta \DualNorm{s}$, then $\phi'(0) \geq 0$, so $0$ is a solution of
  \eqref{aux-prob}. Note that this solution is unique if $\DualPairing{a}{s} <
  \beta \DualNorm{s}$ because then $\phi'(0) > 0$, i.e. $\phi$ is strictly
  increasing on $\NonnegativeRay$.

  Suppose $\DualPairing{a}{s} > \beta \DualNorm{s}$. Then, $\beta <
  \DualNorm{a}$ and thus $\AbsoluteValue{\beta} < \DualNorm{a}$. Note that, for
  any $\tau \geq 0$, we have $\phi(\tau) \geq \tau (\DualNorm{a} + \beta) -
  \DualNorm{s}$. Hence, the sublevel sets of $\phi$, intersected with
  $\NonnegativeRay$, are bounded, so problem~\eqref{aux-prob} has a solution.
  Since $\phi'(0) < 0$, any solution of \eqref{aux-prob} is strictly positive
  and so must be a solution of problem~\eqref{aux-uncon} for $A \DefinedEqual a$
  and $b \DefinedEqual \beta$. But, by \cref{lm-uncon-min}, the latter solution
  is unique and equals $u(H, s, a, \beta)$.
  
  We have proved that \cref{opt-tau} is indeed a solution of \eqref{aux-prob}.
  Moreover, when $\DualPairing{a}{s} \neq \beta \DualNorm{s}$, we have shown
  that this solution is unique. It remains to prove the uniqueness of solution
  when $\DualPairing{a}{s} = \beta \DualNorm{s}$, assuming additionally that
  $\beta < \DualNorm{a}$. But this is simple. Indeed, by our assumptions,
  $\AbsoluteValue{\beta} < \DualNorm{a}$, so $\AbsoluteValue{\DualPairing{a}{s}}
  = \AbsoluteValue{\beta} \DualNorm{s} < \DualNorm{a} \DualNorm{s}$. Hence, $a$
  and $s$ are linearly independent. But then $\phi$ is strictly convex and thus
  its minimizer is unique.
\end{proof}

\subsection{Computation of Dual Multipliers}

In this section, we prove the correctness of \cref{alg-dual-mult}.

For $s \in \VectorSpace{E}\Dual$, let $X(s)$ be the subdifferential of
$\DualNorm{\cdot}$ at the point $s$:
\begin{equation}\label{def-set-X}
  X(s) \DefinedEqual \begin{cases}
    \Set{H s / \DualNorm{s}}, & \text{if $s \neq 0$}, \\
    B_0, & \text{if $s = 0$}.
  \end{cases}
\end{equation}
Clearly, $X(s) \subseteq B_0$ for any $s \in \VectorSpace{E}\Dual$. When $s \neq
0$, we denote the unique element of $X(s)$ by $x(s)$.

Let us formulate a convenient optimality condition.

\begin{lemma}\label{lm-opt-cond}%
  Let $A$ be the linear operator from $\RealField^m$ to $\VectorSpace{E}\Dual$,
  defined by $A u \DefinedEqual \sum_{i=1}^m u_i a_i$, where $a_1, \dots, a_m
  \in \VectorSpace{E}\Dual$, and let $b \in \RealField^m$, $s \in
  \VectorSpace{E}\Dual$. Then, $\mu^* \in \NonnegativeRay^m$ is a minimizer of
  $\psi(\mu) \DefinedEqual \DualNorm{s - A \mu} + \DualPairing{\mu}{b}$ over
  $\NonnegativeRay^m$ if and only if $X(s - A \mu^*) \cap L_1(\mu_1^*) \dots
  L_m(\mu_m^*) \neq \emptyset$, where, for each $1 \leq i \leq m$ and $\tau >
  0$, we denote $L_i(\tau) \DefinedEqual \SetBuilder{x \in
  \VectorSpace{E}}{\DualPairing{a_i}{x} \leq b_i}$, if $\tau = 0$, and
  $L_i(\tau) \DefinedEqual \SetBuilder{x \in
  \VectorSpace{E}}{\DualPairing{a_i}{x} = b_i}$, if $\tau > 0$.
\end{lemma}

\begin{proof}
  Indeed, the standard optimality condition for a convex function over the
  nonnegative orthant is as follows: $\mu^* \in \NonnegativeRay^m$ is a
  minimizer of $\psi$ on $\NonnegativeRay^m$ if and only if there exists $g^*
  \in \Subdifferential \psi(\mu^*)$ such that $g_i^* \geq 0$ and $g_i^* \mu_i^*
  = 0$ for all $1 \leq i \leq m$. It remains to note that $\Subdifferential
  \psi(\mu^*) = b - A\Adjoint X(s - A \mu^*)$.
\end{proof}

\begin{theorem}\label{th-alg-dual-mult-cor}%
  \Cref{alg-dual-mult} is well-defined and returns a solution
  of~\eqref{dual-prob}.
\end{theorem}

\begin{proof}
  \ProofPart

  For each $i = 1, 2$ and $\tau \geq 0$, denote $L_i^- \DefinedEqual
  \SetBuilder{x \in \VectorSpace{E}}{\DualPairing{a_i}{x} \leq b_i}$, $L_i
  \DefinedEqual \SetBuilder{x \in \VectorSpace{E}}{\DualPairing{a_i}{x} = b_i}$,
  $L_i(\tau) \DefinedEqual L_i^-$, if $\tau = 0$, and $L_i(\tau) \DefinedEqual
  L_i$, if $\tau > 0$.

  \ProofPart

  From \cref{Slater,lm-aux-op}, it follows that \cref{step:comp-tau} is
  well-defined and, for each $i = 1, 2$, $\tau_i$ is a solution
  of~\eqref{aux-prob} with parameters $(s, a_i, b_i)$. Hence, by
  \cref{lm-opt-cond},
  \begin{equation}\label{opt-cond-sep}
    X(s - \tau_i a_i) \cap L_i(\tau_i) \neq \emptyset,
    \quad i = 1, 2.
  \end{equation}

  \ProofPart

  Consider \cref{step:check-set-inc}. Note that the condition $\xi_1 \leq b_2$
  is equivalent to $B_0 \cap L_1^- \subseteq L_2^-$ since $\xi_1 = \max_{x \in
  B_0 \cap L_1^-} \DualPairing{a_2}{x}$. If $B_0 \cap L_1^- \subseteq L_2^-$,
  then, by \cref{opt-cond-sep}, $X(s - \tau_1 a_1) \cap L_1(\tau_1) \cap L_2^- =
  X(s - \tau_1 a_1) \cap L_1(\tau_1) \neq \emptyset$, so, by \cref{lm-opt-cond},
  $(\tau_1, 0)$ is indeed a solution of \eqref{dual-prob}.
  
  Similarly, if $\xi_2 \leq b_1$, then $B_0 \cap L_2^- \subseteq L_1^-$ and $(0,
  \tau_2)$ is a solution of \eqref{dual-prob}.

  \ProofPart

  From now on, we can assume that $B_0 \cap L_1^- \cap \Interior L_2^+ \neq
  \emptyset$, $B_0 \cap L_2^- \cap \Interior L_1^+ \neq \emptyset$, where
  $\Interior L_i^+ \DefinedEqual \SetBuilder{x \in
  \VectorSpace{E}}{\DualPairing{a_i}{x} > b_i}$, $i = 1, 2$. Combining this with
  \cref{Slater}, we obtain\footnote{%
    Take an appropriate convex combination of two points from the specified
  nonempty convex sets.%
  }%
  \begin{equation}\label{sect-inter-halfsp}
    \Interior B_0 \cap L_1 \cap L_2^- \neq \emptyset,
    \qquad
    \Interior B_0 \cap L_2 \cap L_1^- \neq \emptyset.
  \end{equation}

  Suppose $\DualPairing{a_2}{H (s - \tau_1 a_1)} \leq b_2 \DualNorm{s - \tau_1
  a_1}$ at \cref{step:check-points}. 1) If $s \neq \tau_1 a_1$, then $X(s -
  \tau_1 a_1)$ is a singleton, $x(s - \tau_1 a_1) = H (s - \tau_1 a_1) /
  \DualNorm{s - \tau_1 a_1}$, so we obtain $x(s - \tau_1 a_1) \in L_2^-$.
  Combining this with \cref{opt-cond-sep}, we get $x(s - \tau_1 a_1) \in
  L_1(\tau_1) \cap L_2^-$. 2) If $s = \tau_1 a_1$, then $X(s - \tau_1 a_1) \cap
  L_1(\tau_1) \cap L_2^- = B_0 \cap L_1(\tau_1) \cap L_2^- \neq \emptyset$ in
  view of the first claim in \cref{sect-inter-halfsp} (recall that $L_1
  \subseteq L_1(\tau_1)$). Thus, in any case, $X(s - \tau a_1) \cap L_1(\tau_1)
  \cap L_2^- \neq \emptyset$, and so, by \cref{lm-opt-cond}, $(\tau_1, 0)$ is a
  solution of \eqref{dual-prob}.
    
  Similarly, one can consider the case when $\DualPairing{a_1}{H (s - \tau_2
  a_2)} \leq b_1 \DualNorm{s - \tau_2 a_2}$ at \cref{step:check-points}.

  \ProofPart

  Suppose we have reached \cref{step:final}. From now on, we can assume that
  \begin{equation}\label{check-points-failed}
    X(s - \tau_1 a_1) \cap L_1(\tau_1) \cap \Interior L_2^+ \neq \emptyset,
    \qquad
    X(s - \tau_2 a_2) \cap L_2(\tau_2) \cap \Interior L_1^+ \neq \emptyset.
  \end{equation}
  Indeed, since both conditions at \cref{step:check-points} have not been
  satisfied, $s \neq \tau_i a_i$, $i = 1, 2$, and $x(s - \tau_1 a_1) \notin
  L_2^-$, $x(s - \tau_2 a_2) \notin L_1^-$. Also, by \cref{opt-cond-sep}, $x(s -
  \tau_i a_i) \in L_i(\tau_i)$, $i = 1, 2$.

  Let $\mu \in \NonnegativeRay^2$ be any solution of \eqref{dual-prob}. By
  \cref{lm-opt-cond}, $X(s - A \mu) \cap L_1(\mu_1) \cap L_2(\mu_2) \neq
  \emptyset$. Note that we cannot have $\mu_2 = 0$. Indeed, otherwise, we get
  $X(s - \mu_1 a_1) \cap L_1(\mu_1) \cap L_2^- \neq \emptyset$, so $\mu_1$ must
  be a solution of \eqref{aux-prob} with parameters $(s, a_1, b_1)$. But, by
  \cref{lm-aux-op}, such a solution is unique (in view of the second claim in
  \cref{check-points-failed}, $\DualPairing{a_1}{x} > b_1$ for some $x \in B_0$,
  so $b_1 < \DualNorm{a_1}$). Hence, $\mu_1 = \tau_1$, and we obtain a
  contradiction with \cref{check-points-failed}. Similarly, we can show that
  $\mu_1 \neq 0$. Consequently, $\mu_1, \mu_2 > 0$, which means that $\mu$ is a
  solution of \eqref{aux-uncon}.
  
  Thus, at this point, any solution of \eqref{dual-prob} must be a solution of
  \eqref{aux-uncon}. In view of \cref{lm-uncon-min}, to finish the proof, it
  remains to show that $a_1$, $a_2$ are linearly independent and
  $\DualPairing{b}{(A\Adjoint H A)^{-1} b} < 1$. But this is simple. Indeed,
  from \cref{check-points-failed}, it follows that
  \begin{equation}\label{sect-inter-int-halfsp}
    \text{either} \quad
    B_0 \cap L_1 \cap \Interior L_2^+ \neq \emptyset
    \quad \text{or} \quad
    B_0 \cap L_2 \cap \Interior L_1^+ \neq \emptyset
  \end{equation}
  since $\tau_1$ and $\tau_2$ cannot both be equal to $0$. Combining
  \cref{sect-inter-int-halfsp,sect-inter-halfsp}, we see that $\Interior B_0
  \cap L_1 \cap L_2 \neq \emptyset$ and, in particular, $L_1 \cap L_2 \neq
  \emptyset$. Hence, $a_1$, $a_2$ are linearly independent (otherwise, $L_1 =
  L_2$, which contradicts \cref{sect-inter-int-halfsp}). Taking any $x \in
  \Interior B_0 \cap L_1 \cap L_2$, we obtain $\Norm{x} < 1$ and $A\Adjoint x =
  b$, hence $\DualPairing{b}{(A\Adjoint H A)^{-1} b} = \DualPairing{A\Adjoint
  x}{(A\Adjoint H A)^{-1} A\Adjoint x} \leq \Norm{x}^2 < 1$, where we have used
  $A (A\Adjoint H A)^{-1} A\Adjoint \preceq H^{-1}$.
\end{proof}

%%%%%%%%%%%%%%%%%%%%%%%%%%%%%%%%%%%%%%%%%%%%%%%%%%%%%%%%%%%%%%%%%%%%%%%%%%%%%%%%
%%%%%%%%%%%%%%%%%%%%%%%%%%%%%%%%%%%%%%%%%%%%%%%%%%%%%%%%%%%%%%%%%%%%%%%%%%%%%%%%
\printbibliography
%%%%%%%%%%%%%%%%%%%%%%%%%%%%%%%%%%%%%%%%%%%%%%%%%%%%%%%%%%%%%%%%%%%%%%%%%%%%%%%%
%%%%%%%%%%%%%%%%%%%%%%%%%%%%%%%%%%%%%%%%%%%%%%%%%%%%%%%%%%%%%%%%%%%%%%%%%%%%%%%%

\end{document}